\documentclass[10pt]{article}

\usepackage[english]{babel}
\usepackage[utf8]{inputenc}
\usepackage[T1]{fontenc}

\usepackage{amsmath, amsthm, amssymb, mathtools}

\usepackage[shortlabels]{enumitem}

\usepackage{cases}

\usepackage{sectsty}

\usepackage[a4paper]{geometry}

\usepackage[displaymath, mathlines]{lineno}

\AtBeginDocument{\def\MR#1{}}

\usepackage{hyperref}

\providecommand{\keywords}[1]{\textbf{Keywords.} #1}
\providecommand{\MSC}[1]{\textbf{2010 Mathematics Subject Classification.} #1}

\newtheorem{theorem}{Theorem}[section]
\newtheorem{corollary}[theorem]{Corollary}
\newtheorem{lemma}[theorem]{Lemma}
\newtheorem{proposition}[theorem]{Proposition}
\newtheorem{example}[theorem]{Example}

\theoremstyle{definition}
\newtheorem{definition}[theorem]{Definition}
\newtheorem{remark}[theorem]{Remark}

\usepackage{tikz}
\usetikzlibrary{positioning}
\usetikzlibrary{decorations.pathreplacing}
\usepackage{caption}

\renewcommand\epsilon{\varepsilon}
\renewcommand\mapsto{\longmapsto}

\usepackage{color}

\usepackage{xspace}

\newcommand{\R}{\field{R}\xspace}

\newcommand{\N}{\field{N}\xspace}

\newcommand{\field}[1]{\ensuremath{\mathbb{#1}}}

\newcommand{\ens}[1]{ \left\{#1\right\} }

\newcommand\diag{\mathrm{diag} \,}

\newcommand{\Tinf}{T_{\mathrm{inf}}}

\newcommand\pt[1]{\frac{\partial #1}{\partial t}}
\newcommand\px[1]{\frac{\partial #1}{\partial x}}
\newcommand\pxx[1]{\frac{\partial^2 #1}{\partial x^2}}
\newcommand\pxi[1]{\frac{\partial #1}{\partial \xi}}

\newcommand\pxixi[1]{\frac{\partial^2 #1}{\partial \xi^2}}

\newcommand\Tau{\mathcal{T}}

\newcommand{\rank}{\mathrm{rank} \,}
\newcommand{\Span}{\mathrm{span} \,}

\newcommand\Id{\mathrm{Id}}

\def\norm#1{\left\|#1\right\|}
\newcommand\abs[1]{\left|#1\right|}
\newcommand\st{\quad \middle| \quad}

\newcommand{\ps}[3]{ {\left\langle #1 , #2 \right\rangle}_{#3} }

\newcommand\dds{\frac{d}{ds}}

\usepackage{ifthen}
\newcommand{\syst}[2]{
\ifthenelse{\equal{#2}{}}{\left(\Lambda,#1,Q\right)}
{\ifthenelse{\equal{#2}{b}}{\left(\Lambda,-,Q,#1\right)}{}}
{\ifthenelse{\equal{#2}{c}}{\left(\Lambda,-,Q^0,#1\right)}{}}
}

\newcommand\clos[1]{\overline{#1}}

\newcommand\newc{\bar{c}}
\newcommand\newr{\bar{r}}

\newcommand\kal[1]{\mathrm{Kal}\left(#1\right)}

\newcommand\opP{\mathcal{P}}
\newcommand\opQ{\mathcal{Q}}

\newcommand\newm{\alpha}
\newcommand\newq{\beta}

\newcommand\Mset{\R^{n \times n}_0}

\newcommand\newk{k^{\mathfrak{c}}}
\newcommand\newn{n^{\mathfrak{c}}}
\newcommand\seq{\psi}
\newcommand\seqset{(\seq^r )_{r \in \N}}

\newcommand\knoi{k^{\mathfrak{c}}}
\newcommand\fnoi{f^{\mathfrak{c}}}
\newcommand\sigmanoi{\sigma^{\mathfrak{c}}}
\newcommand\zetanoi{\zeta^{\mathfrak{c}}}
\newcommand\Mnoi{M^{\mathfrak{c}}}
\newcommand\mnoi{m^{\mathfrak{c}}}

\newcommand\tr{{\mathsf{T}}}

\newcommand\qshift{\hat{q}}

\newcommand\Ezero{E_0}
\newcommand\Eone{E_1}
\newcommand\Etwo{E_2}

\newcommand\lambdan{\bar{\lambda}}
\newcommand\mn{\bar{m}}

\title{Minimal null control time of some 1D linear hyperbolic balance laws with constant coefficients and properties of related kernel equations}

\author{
Long Hu\thanks{School of Mathematics, Shandong University, Jinan, Shandong 250100, China.  E-mail: \texttt{hul@sdu.edu.cn}}
\and
Guillaume Olive\thanks{Faculty of Mathematics and Computer Science, Jagiellonian University, ul. {\L}ojasiewicza 6, 30-348 Krak\'{o}w, Poland. E-mail: \texttt{math.golive@gmail.com} or \texttt{guillaume.olive@uj.edu.pl}}
}

\date{}

\begin{document}

\maketitle

\begin{abstract}
In this work, we study the null controllability by one-sided boundary controls of one-dimensional hyperbolic balance laws with constant coefficients.
Our first result shows that, when the system has only one negative or positive speed, the minimal null control time of such systems depends on some orthogonality conditions for a particular sequence.
This sequence is explicit in function of the coefficients of the system but it is defined by a nonlinear recurrence relation.
Our second result then completes the previous one by giving explicit bounds on the number of orthogonality conditions that have to be checked in two nontrivial situations.
The proofs rely on a careful analysis of the so-called kernel equations associated with the system, including a new well-posedness result.
Our results are also valid for the finite-time stabilization property.
\end{abstract}


\vspace{0.2cm}
\MSC{35L40, 93B05, 93D15, 11B}

\section{Introduction and main results}

\subsection{Problem description}

In this paper, we are interested in the null controllability properties of a class of one-dimensional (1D) hyperbolic system with constant coefficients (see e.g. \cite[Chap. 1]{BC16} for applications).
The equations describing such phenomenons are
\begin{subequations}\label{syst}
\begin{equation}\label{syst:equ}
\pt{y}(t,x)+\Lambda \px{y}(t,x)=My(t,x).
\end{equation}
Above, $t \in (0,T)$ is the time variable, $T>0$, $x \in (0,1)$ is the space variable and the state is $y:(0,T) \times (0,1) \to \R^n$ $(n \geq 2$).
The matrix $\Lambda \in \R^{n \times n}$ will always be assumed diagonal $\Lambda =\diag(\lambda_1,\ldots,\lambda_n)$, with $m \geq 1$ negative speeds and $p \geq 1$ positive speeds ($m+p=n$):
\begin{equation}\label{hyp speeds}
\lambda_1<\cdots<\lambda_m <0<\lambda_{m+1}<\cdots<\lambda_{m+p}.
\end{equation}
The matrix $M \in \R^{n \times n}$ couples the equations of the system inside the domain and will be called the internal coupling matrix.
We will consider an initial condition at time $t=0$:
\begin{equation}\label{syst:ID}
y(0,x)=y^0(x).
\end{equation}

Let us now discuss the boundary conditions.
The structure of $\Lambda$ induces a natural splitting of the state into components corresponding to negative and positive speeds, denoted respectively by $y_-$ and $y_+$.
For the above system to be well-posed in $(0,T) \times (0,1)$ with an initial condition at time $t=0$, we then need to add boundary conditions at $x=1$ for $y_-$ and at $x=0$ for $y_+$.
We will consider the following type of boundary conditions, motivated by its numerous applications (see again \cite{BC16}):
\begin{equation}\label{syst:BC}
y_-(t,1)=u(t), \quad y_+(t,0)=Qy_-(t,0).
\end{equation}
The function $u$ is the so-called control, it will be at our disposal.
It only acts on one part of the boundary and, on the other part of the boundary, the equations are coupled by the matrix $Q \in \R^{p \times m}$. This matrix will be called the boundary coupling matrix.
\end{subequations}
In what follows, \eqref{syst:equ}, \eqref{syst:ID} and \eqref{syst:BC} together will be referred to as system \eqref{syst}.

We recall that system \eqref{syst} is well-posed in the following functional setting: for every $T>0$, $y^0 \in L^2(0,1)^n$ and $u \in L^2(0,T)^m$, there exists a unique solution $y$ to system \eqref{syst} with
$$
y \in C^0([0,T];L^2(0,1)^n) \cap C^0([0,1];L^2(0,T)^n).
$$
By solution we mean ``solution along the characteristics''.
We refer for instance to \cite{CHOS21} for a proof of this well-posedness result in such a setting (see also \cite[Appendix A]{BC16} when $u=0$).

The regularity $C^0([0,T];L^2(0,1)^n)$ of the solution allows us to consider control problems in the space $L^2(0,1)^n$:

\begin{definition}
Let $T>0$ be fixed.
We say that system \eqref{syst} is null controllable in time $T$ if, for every $y^0 \in L^2(0,1)^n$, there exists $u \in L^2(0,T)^m$ such that the corresponding solution $y$ to system \eqref{syst} satisfies
$$y(T,\cdot)=0.$$
\end{definition}

Since controllability in time $T_1$ implies controllability in any time $T_2 \geq T_1$, it is natural to try to find the smallest possible control time, the so-called ``minimal control time''.

\begin{definition}
For any $\Lambda,M$ and $Q$ as above, we denote by $\Tinf(\Lambda,M,Q) \in [0,+\infty]$ the minimal null control time of system \eqref{syst}, that is
\begin{equation}\label{def min cont time}
\Tinf(\Lambda,M,Q)=
\inf\ens{T>0 \st \text{System \eqref{syst} is null controllable in time $T$}}.
\end{equation}
\end{definition}

The time $\Tinf(\Lambda,M,Q)$ is named ``minimal'' null control time according to the current literature, despite it is not always a minimal element of the set.
We keep this naming here, but we use the notation with the ``inf'' to avoid eventual confusions.
The goal of this article is to characterize $\Tinf(\Lambda,M,Q)$ in some new situations.

In order to state our results and those of the literature, we need to introduce the following times:
$$
T_i=\frac{1}{-\lambda_i} \quad \text{ if } i \leq m, \quad
T_i=\frac{1}{\lambda_i} \quad \text{ if } i \geq m+1.
$$
The time $T_i$ is the time needed for the controllability of a single equation (the transport equation) with speed $\lambda_i$.
Note that the assumption \eqref{hyp speeds} implies in particular the following order relation:
\begin{equation}\label{order times}
\begin{dcases}
T_1 \leq \cdots \leq T_m, \\
T_n \leq \cdots \leq T_{m+1}.
\end{dcases}
\end{equation}

\subsection{Literature}

Here, we briefly describe the results of the literature about the null controllability of system \eqref{syst}.
All the results below are also valid for space-dependent versions of this system.

\begin{itemize}
\item
It was first proved in the celebrated survey \cite{Rus78} that system \eqref{syst} is null controllable in any time $T \geq T_{m+1}+T_m$.
A strength of this result is that it is valid for any $M$ and $Q$.
However, it was also observed in that paper that the minimal control time can be smaller than $T_{m+1}+T_m$.
Finding the minimal control time even in the simpler case $M=0$ was then left as an open problem.

\item
For $M=0$, the minimal null control time was eventually found in \cite{Wec82}.
The author gave an explicit expression of this time in terms of some indices related to $Q$.

\item
Finding the minimal null control time for arbitrary $M$ and $Q$ is still an open challenging problem.
There has been a recent resurgence on the characterization of this time.
A first result in this direction was obtained in \cite{CN19} and then completed in \cite{CN21}.
Therein, the authors introduced a class of boundary coupling matrices $Q$ for which they showed that the minimal null control time is smaller than $T_{m+1}+T_m$, whatever the internal coupling matrix $M$ is.

\item
For full row rank boundary coupling matrices ($\rank Q=p$), the minimal null control time was found in \cite{HO21-JMPA}.
In this case, it has been shown that this time is the same as for the system without internal coupling ($M=0$).

\item
For systems of $n=2$ equations, the minimal null control time was found in \cite{CVKB13} and \cite{HO21-COCV}.
In particular, it has been shown in the second reference that this time depends on the internal coupling matrix $M$ when the boundary coupling matrix is zero.
This is a feature that was not observed nor highlighted in all the other works and that the results of the present paper will also share.

\item
Finally, the smallest and largest values that the minimal null control time can take with respect to the internal coupling matrix $M$ were found in \cite{HO22-JDE}.

\end{itemize}

Other related works include for instance \cite{CHOS21,CN21-pre,MAK22} about time-dependent versions of system \eqref{syst} and \cite{Li10,LR10,Hu15,CN20,CN22} for quasilinear versions of this system (in a $C^1$ framework).

\subsection{Notations and important definitions}\label{sect nota 1}

To state the main results of this article we first need to introduce some notations.

All along this article, we denote by $A^\tr$ the transpose of a matrix $A$.
For any integer $N \geq 1$, $\R^{N \times N}_0$ denotes the set of matrices of size $N \times N$ with diagonal entries all equal to zero.
The matrix $\Id_N$ denotes the $N \times N$ identity matrix.
A matrix (or matrix-valued function) of size $N_1 \times N_2$ will simply be denoted using the corresponding lowercase letter when $N_2=1$ (e.g. $Q \in \R^{p \times m}$ will be denoted by $q \in \R^p$ when $m=1$).
The inner product of two vectors $v_1,v_2 \in \R^{n-1}$ will be denoted by $\ps{v_1}{v_2}{}$.

Let us now introduce a sequence that will play a key role throughout this paper.
For any $i \in \ens{1,\ldots,n}$, we first define the following quantities.

\begin{itemize}
\item
For every $r,j \in \ens{1,\ldots,n}$, we denote by
$$
\newm_{rj}=
\frac{m_{rj}}{\lambda_i-\lambda_j} \quad \text{ if } j \neq i,
\quad
\newm_{ri}=\frac{m_{ri}}{\lambda_i}.
$$

\item
Let then $A, D \in \R^{(n-1) \times (n-1)}$ be the matrices defined by
$$
A=(\newm_{rj})_{r,j \neq i},
\quad
D=\diag \left(\frac{\lambda_i+\lambda_j}{\lambda_i-\lambda_j}\right)_{j \neq i},
$$
and let $\seq^0, w \in \R^{n-1}$ be the vectors defined by
$$
\seq^0=(\newm_{ij})_{j \neq i},
\quad
v=\left(-\frac{1}{2}\newm_{ji}\right)_{j \neq i}.
$$

\item
Let $\seqset \subset \R^{n-1}$ be the sequence defined by $\seq^0$ and
\begin{equation}\label{def seq}
\seq^1= A^\tr\seq^0, \quad
\seq^r =
A^\tr\seq^{r-1}
+D \sum_{\ell=0}^{r-2} \ps{v}{\seq^{r-2-\ell}}{} \seq^{\ell}, \quad \forall r \geq 2.
\end{equation}

\item
Finally, for $q \in \R^{n-1}$, let $b \in \R^{n-1}$ be the nonzero vector defined by
\begin{equation}\label{def b}
b=(\newq_j)_{j \neq i}, \quad \newq=-\Lambda\begin{pmatrix} 1 \\ q \end{pmatrix}.
\end{equation}

\end{itemize}
All the above quantities depend on the index $i$ but we omit it for clarity.

With the previous notations, we introduce the following sets.
For any $k \in \ens{2,\ldots,n+1}$, we denote by $\mathcal{C}_k$ the set of $(M,q) \in \Mset \times \R^{n-1}$ such that, for every $2 \leq i<k$, we have $q_{i-1}=0$ and
\begin{equation}\label{cond seq}
\ps{b}{\seq^r}{}=0, \quad \forall r \in \N. 
\end{equation}
Here, we use the convention that $\mathcal{C}_2=\Mset \times \R^{n-1}$.
Additionally, we will denote by $\mathcal{C}_{n+2}=\emptyset$.
Note that we then have $\mathcal{C}_2 \supset \mathcal{C}_3 \supset \cdots \supset \mathcal{C}_{n+1} \supset \mathcal{C}_{n+2}$.

\subsection{Main results and comments}

The first result of this article is the following characterization of the minimal null control time in the case of one negative speed.

\begin{theorem}\label{thm M}
Assume that $m=1$.
Let us denote by
\begin{equation}\label{def times tau}
\tau_i=\max\ens{T_1+T_i,T_2} \quad \text{ if } 2 \leq i \leq n,
\quad
\tau_{n+1}=\max\ens{T_1, T_2},
\end{equation}
(we have $\tau_2 \geq \tau_3 \geq \cdots \geq \tau_{n+1}$ from \eqref{order times}).
Then, for any $M \in \Mset$ and $q \in \R^{n-1}$, we have:
\begin{enumerate}
\item
$\Tinf(\Lambda,M,q) \in \ens{\tau_2,\ldots,\tau_{n+1}}$.
Moreover, the infimum is always reached (in \eqref{def min cont time}).

\item
For any $k \in \ens{2,\ldots,n+1}$, we have
$$
\Tinf(\Lambda,M,q)=\tau_k \quad \Longleftrightarrow \quad (M,q) \in \mathcal{C}_k \setminus \mathcal{C}_{k+1}.
$$

\end{enumerate}

\end{theorem}

We recall that the set $\mathcal{C}_k$ is defined at the end of Section \ref{sect nota 1}.

\begin{remark}
Theorem \ref{thm M} solves the open problem raised at the end of \cite[Section 5]{HO21-COCV} for systems with constant coefficients.
\end{remark}

\begin{remark}
Theorem \ref{thm M} remains valid if we replace everywhere in this article the null controllability property by the finite-time stabilization property by $L^2$ bounded feedbacks (that is when the control $u$ is looked under the more particular form $u(t)=\int_0^1 r(\xi)^\tr y(t,\xi) \, d\xi$ with $r \in L^2(0,1)^n$).
This easily follows from the proof below.
\end{remark}

Even though the set $\mathcal{C}_k$ is explicit in function of the parameters of the system, the orthogonality conditions \eqref{cond seq} that define this set can be difficult to study in general because the sequence $\seqset$ is defined by a nonlinear recurrence relation.
Note however that there always exists an integer $N \geq 1$ such that
$$\seq^r \in \Span \ens{\seq^s \st s \leq N-1}, \quad \forall r \geq N,$$
so that
\begin{equation}\label{cond seq N}
\ps{b}{\seq^r}{}=0, \quad \forall r \in \N
\quad \Longleftrightarrow \quad
\ps{b}{\seq^r}{}=0, \quad \forall r \in \ens{0,\ldots,N-1}.
\end{equation}
This means that \eqref{cond seq} only needs to be checked for the first $N$ values of $r$.
However, such a $N$ depends on the sequence and it is a priori unknown, so that, in practice, we do not know when we have to stop checking the orthogonality conditions.
Our second result provides information on this issue in two particular cases:

\begin{theorem}\label{thm bounds Nseq}
Let $i \in \ens{1,\ldots,n}$ be fixed.
Define
$$
N_\seq=\min \ens{N \geq 1 \st \text{\eqref{cond seq N} holds}}.
$$
We have $N_\seq \leq 3$ for $n=3$ and $N_\seq \leq 6$ for $n=4$.
\end{theorem}

\begin{remark}
It would be interesting to find a bound of $N_\seq$ with respect to $n$ for arbitrary $n$.
\end{remark}

Our main results can for instance be combined to deduce a very explicit characterization of the minimal null control time in the following particular case:

\begin{corollary}
Assume that $m=1$ and $p=2$.
Then, for any $M \in \R^{3 \times 3}_0$ and $q \in \R^2$, we have:
\begin{enumerate}
\item
$\Tinf(\Lambda,M,q)=\max\ens{T_1,T_2}$ if, and only if, $(M,q)$ satisfies
$$
q=0,
\quad
m_{21}=m_{31}=0.
$$

\item
$\Tinf(\Lambda,M,q)=\max\ens{T_1+T_3,T_2}$ if, and only if, $(M,q)$ satisfies
$$
q=0, \quad m_{21}=m_{23}=0, \quad m_{31} \neq 0,
$$
or
$$
q_1=0, \quad q_2 \neq 0
\quad \text{ and } \quad
\left(
m_{21}=m_{23}=0
\quad \text{ or } \quad
\begin{dcases}
m_{21}=rs m_{23}, \\
m_{31}= r^2 s m_{13}, \\
m_{32}= -r m_{12},
\end{dcases}
\right),
$$
where $r=-\frac{\lambda_3 q_2}{\lambda_1}$ and $s=\frac{\lambda_2-\lambda_1}{\lambda_2-\lambda_3}$.

\item
In all the other situations, $\Tinf(\Lambda,M,q)=T_1+T_2$.

\end{enumerate}
\end{corollary}

For $p=3$, there is no simple presentation as for $p=2$, even though the orthogonality conditions are explicit (see also Remark \ref{rem explicit cond C} below) and we know that we only have to check the first six conditions.
Therefore, we only give a nontrivial example:

\begin{example}\label{example p=3}
Let $\sigma \in \R \setminus\ens{0}$ be arbitrary and consider system \eqref{syst} with
$$
\Lambda=\diag(-1,1,2,3),
\quad
M=
\begin{pmatrix}
0 & -3 & 1/(2\sigma) & 0 \\
2 & 0 & 0 & -2 \\
3\sigma & 0 & 0 & -\sigma \\
0 & 2 & 1/(2\sigma) & 0
\end{pmatrix},
\quad
q=
\begin{pmatrix}
0 \\
0 \\
1/3
\end{pmatrix}.
$$

\begin{itemize}
\item
For $i=2$, we have $b=\begin{pmatrix}
1 &
0 &
-1
\end{pmatrix}^\tr$
and
$$
\seq^0=\begin{pmatrix}
1 \\
0 \\
1
\end{pmatrix},
\quad
\seq^1=-\frac{1}{\sigma} \begin{pmatrix}
0 \\
1 \\
0
\end{pmatrix},
\quad
\seq^2=-\frac{3}{2} \seq^0,
\quad
\seq^3=-3 \seq^1,
\quad
\seq^4=\frac{9}{2} \seq^0,
\quad
\seq^5=\frac{45}{4} \seq^1.
$$

\item
For $i=3$, we have the same $b$ and
$$
\seq^0=\sigma \begin{pmatrix}
1 \\
0 \\
1
\end{pmatrix},
\quad
\seq^1=-\sigma \begin{pmatrix}
0 \\
1 \\
0
\end{pmatrix},
\quad
\seq^2=-\frac{3}{4} \seq^0,
\quad
\seq^3=-\frac{3}{2} \seq^1,
\quad
\seq^4=\frac{9}{8} \seq^0,
\quad
\seq^5=\frac{45}{16} \seq^1.
$$

\end{itemize}
From Theorems \ref{thm M} and \ref{thm bounds Nseq}, we deduce that $\Tinf(\Lambda,M,q)=\tau_4=4/3$.

\end{example}

\begin{remark}\label{rem explicit cond C}
For arbitrary $n$, we will see that the orthogonality conditions \eqref{cond seq} are satisfied if one of the following three conditions holds:
\begin{enumerate}[(C1)]
\item\label{n cond 1}
$\kal{A,v}^\tr  \seq^0=\kal{A,b}^\tr  \seq^0=0$, where $\kal{A,h}=(h|Ah|\cdots|A^{n-2}h) \in \R^{(n-1) \times (n-1)}$ denotes the Kalman matrix of $(A,h)$, for any $h \in \R^{n-1}$.

\item\label{n cond 2}
There exists $\emptyset \neq J \subsetneq \ens{1,\ldots,n-1}$ such that $\seq^0_j=a_{rj}=b_r=0$ for every $j \not\in J$ and $r \in J$.

\item\label{n cond 3}
$\ps{b}{\seq^0}{}=0$ and there exists $j_0 \in \ens{1,\ldots,n-1}$ such that $b_{j_0}=0$ and $\rank \Delta_{j_0}=1$, where
$$
\Delta_{j_0}
=\begin{pmatrix}
(D-d_{j_0}) \seq^0 & A^\tr e_{j_0} \\
A^\tr \seq^0 & v_{j_0} \seq^0 -\ps{v}{\seq^0}{} e_{j_0}
\end{pmatrix}
\in \R^{2(n-1) \times 2},
$$
where $d_{j_0}$ is the $j_0$-th diagonal entry of $D$ and $e_{j_0}$ is the $j_0$-th canonical vector of $\R^{n-1}$.

\end{enumerate}
We will also see that, for $n=3$ (resp. $n=4$), it is necessary that one of the conditions \ref{n cond 1}, \ref{n cond 2} (resp. \ref{n cond 1}, \ref{n cond 2}, \ref{n cond 3}) holds (it is however preferable to use Theorem \ref{thm bounds Nseq} in these cases).
\end{remark}

The rest of this article is organized as follows.
In Section \ref{sect syst G}, we use the equivalence between the controllability of system \eqref{syst} and that of a simpler system to obtain a characterization of this property in terms of some orthogonality conditions for the derivatives at the origin of any solution to the so-called kernel equations.
In Section \ref{sect deriv K}, we compute these derivatives for a particular solution and we obtain a general formula for this solution.
In Section \ref{sect hyperplane}, we study the orthogonality conditions associated with the previous solution and we deduce our main results.
In Section \ref{sect other explicit kern}, we supplement ours results by studying the structure of the solution associated with the orthogonality conditions.
Finally, in Appendix \ref{app thm G}, we give a simple proof of the characterization of the controllability properties for the equivalent system and, in Appendix \ref{app thm K}, we prove the existence of a solution to the kernel equations by a new approach.

\section{An equivalent system and the kernel equations}\label{sect syst G}

The first step in the proof of our results is to consider a system which is equivalent to our initial system from a control point of view.

\begin{lemma}\label{lem backstepping}
For any $T>0$, system \eqref{syst} is null controllable in time $T$ if, and only if, so is the system
\begin{equation}\label{syst G}
\begin{dcases}
\pt{\tilde{y}}(t,x)+\Lambda \px{\tilde{y}}(t,x)=F(x) \tilde{y}_-(t,0), \\
\tilde{y}_-(t,1)=\tilde{u}(t), \quad \tilde{y}_+(t,0)=Q\tilde{y}_-(t,0),  \\
\tilde{y}(0,x)=\tilde{y}^0(x),
\end{dcases}
\end{equation}
where $F \in C^0([0,1])^{n \times m}$ is defined by
\begin{equation}\label{def G}
F(x)
=
-K(x,0)\Lambda\begin{pmatrix} \Id_m \\ Q \end{pmatrix},
\end{equation}
and $K \in C^0(\clos{\Tau})^{n \times n}$ is any solution to
\begin{equation}\label{kern equ}
\begin{dcases}
\Lambda\px{K}(x,\xi)
+\pxi{K}(x,\xi)\Lambda
+K(x,\xi)M=0,
\\
\Lambda K(x,x)-K(x,x)\Lambda =M,
\end{dcases}
\end{equation}
in the closure of the triangle $\Tau=\ens{(x,\xi) \in \R^2 \st 0<\xi<x<1}$.
\end{lemma}

By solution to \eqref{kern equ} we mean solution along the characteristics.
This result is by now well-known: it consists in using the invertible spatial transformation
$$
\tilde{y}(t,x)=y(t,x)-\int_0^x K(x,\xi)y(t,\xi) \, d\xi,
$$
in order to transform a solution of system \eqref{syst} into a solution of system \eqref{syst G} (see e.g. \cite[Section 2.2]{HVDMK19}).
This idea is the starting point of the so-called backstepping method for partial differential equations and introduced more specifically for hyperbolic systems of two equations in \cite{CVKB13}.
Equations \eqref{kern equ} are thus called the kernel equations.
The difficult point is not so much the result of the previous lemma but rather to prove that \eqref{kern equ} actually has at least a solution.
It follows from the results of \cite{HDMVK16} that there are many solutions to the kernel equations \eqref{kern equ} in $\clos{\Tau}$.

\begin{remark}\label{rem one choice is enough}
The choice of solution to the kernel equations \eqref{kern equ} does not affect the controllability properties of system \eqref{syst} because all the corresponding systems \eqref{syst G}-\eqref{def G} are equivalent from a control point of view.
\end{remark}

Now, two problems naturally arise:
\begin{enumerate}
\item
Can we characterize the null controllability of the equivalent system \eqref{syst G} in function of $\Lambda,Q$ and $F$ ?

\item
If so, can this characterization be presented explicitly in terms of $\Lambda,Q$ and $M$ ?
\end{enumerate}

These problems are still open in general.
One particular case where the first problem has been completely solved is the case $m=1$ (one negative speed).
This was done in \cite[Section 5]{HO21-COCV}.

\begin{theorem}\label{thm HO21}
Assume that $m=1$.
Then, system \eqref{syst G} is null controllable in time $T$ if, and only if,
$$
T \geq \max_{2 \leq i \leq n} \ens{T_1+T_i^*, T_2},
$$
where
$$
T_i^*=
\begin{dcases}
T_i & \text{ if } q_{i-1} \neq 0, \\
T_i (1-\ell(f_i)) & \text{ if } q_{i-1}=0,
\end{dcases}
$$
where $\ell(f_i)=\sup I(f_i)$ with $I(f_i)=\ens{\ell \in (0,1) \st f_i=0 \text{ in } (0,\ell)}$, if $I(f_i) \neq \emptyset$, and $\ell(f_i)=0$ otherwise.
\end{theorem}

The second problem could not be solved though because, even if the conditions for $f$ are explicit, the ``map'' $M \mapsto f$ (``defined'' by \eqref{kern equ}-\eqref{def G}, with $m=1$) is quite complicated.
It was left as an open problem in the same paper.
This is precisely where our main results step in.

From the above result of \cite{HO21-COCV} we see that the values at $x=0$ of $f$ and its derivatives (assuming it is smooth) can affect the minimal null control time $\Tinf(\Lambda,-,q,f)$ of the system ($T_i^*=T_i$ if $f_i^{(N)}(0) \neq 0$ for some $N \geq 0$).
Our idea is to show that these values in fact completely characterize $\Tinf(\Lambda,-,q,f)$ because $M$ is constant and that we can explicitly relate them to $M$ thanks to the kernel equations.

It is clear that $\Tinf(\Lambda,-,q,f)$ is solely characterized by $f(0), f'(0), f''(0),$ etc. if we have
\begin{equation}\label{hyp f analytic}
\text{$f$ is analytic in a neighborhood of $[0,1)$.}
\end{equation}

Under such an assumption, Theorem \ref{thm HO21} takes a simpler form:

\begin{corollary}\label{thm G}
Assume that $m=1$, let $q \in \R^{n-1}$ be given and assume \eqref{hyp f analytic}.
Then, we have:
\begin{enumerate}
\item
$\Tinf(\Lambda,-,q,f) \in \ens{\tau_2,\ldots,\tau_{n+1}}$ (recall \eqref{def times tau}).
Moreover, the infimum is always reached.

\item
For any $k \in \ens{2,\ldots,n+1}$, we have
$$
\Tinf(\Lambda,-,q,f)=\tau_k \quad \Longleftrightarrow \quad (q,f) \in \mathcal{S}_k \setminus \mathcal{S}_{k+1},
$$
where, for every $k \in \ens{2,\ldots,n+1}$, $\mathcal{S}_k$ is the set of $(q,f) \in \R^{n-1} \times C^0([0,1])^n$ such that $q_{i-1}=f_i=0$ for every $2 \leq i<k$ (we use the convention that $\mathcal{S}_2=\R^{n-1} \times C^0([0,1])^n$), and $\mathcal{S}_{n+2}=\emptyset$.

\end{enumerate}

\end{corollary}

This result is immediate from the previous theorem but we give a simple and direct proof in Appendix \ref{app thm G}.
Note that it is the complete analogue of Theorem \ref{thm M} for system \eqref{syst G}.
By Lemma \ref{lem backstepping}, the minimal null control time for the initial system \eqref{syst} is thus also completely determined by the sets $\mathcal{S}_k$.
However, apart from $\mathcal{S}_{n+1}$, these sets are not explicit in terms of $M$, which is unsatisfactory.

Assumption \eqref{hyp f analytic} is indeed satisfied in our framework because we can always find an analytic solution to the kernel equations since $M$ is constant.
More precisely, we have the following result:

\begin{theorem}\label{thm existence and analyticity}
Let $m,p \geq 1$ be arbitrary.
Assume that $M \in \Mset$.
For any $\delta \in \R$ with $\delta \neq 1$, there exists a unique $K \in C^{\infty}(\R^2)^{n \times n}$ that satisfies \eqref{kern equ} for every $(x,\xi) \in \R^2$ and the condition
\begin{equation}\label{general BC}
\diag K(x,\delta x)=0, \quad \forall x \in \R.
\end{equation}
Moreover, it satisfies the estimate
\begin{equation}\label{estim k}
\forall \text{ bounded } V \subset \R^2, \exists C>0, \quad \norm{K}_{C^s(\clos{V})^{n \times n}} \leq C^s,
\quad \forall s \in \N.
\end{equation}

\end{theorem}

The kernel equations \eqref{kern equ} have been extensively studied in the literature (see e.g. \cite{CVKB13,DMVK13,HDM15,HDMVK16,HVDMK19,CN19}) but Theorem \ref{thm existence and analyticity} does not follow from the results contained therein.
The most important difference is that, in Theorem \ref{thm existence and analyticity}, the kernel exists on a larger set than the triangle $\Tau$.
This is crucial since we want $x \mapsto f(x)=-K(x,0)\Lambda \begin{pmatrix} 1 & q \end{pmatrix}^\tr$ to be analytic in an interval of the form $(-\epsilon,1)$, $\epsilon>0$, which does not lie entirely in $\Tau$.
This yields nontrivial issues in the standard fixed point approach, notably because we now have to consider points that are ``on the other side'' of the diagonal $(x,x)$, that is the condition imposed in the kernel equations at $(x,x)$ cannot be consider as a boundary condition anymore.
We have developed in Appendix \ref{app thm K} a new approach to solve the kernel equations that encompasses in particular the proof of Theorem \ref{thm existence and analyticity}.

\begin{remark}
Estimate \eqref{estim k} and Taylor's theorem show that the solution $K$ to \eqref{kern equ}-\eqref{general BC} is in fact a power series.
\end{remark}

As a consequence of Theorem \ref{thm existence and analyticity}, we see that, if $q_{i-1}=0$, then $f=-K(\cdot,0)\Lambda \begin{pmatrix} 1 & q \end{pmatrix}^\tr$ satisfies $f_i=0$ in $(0,1)$ if, and only if,
\begin{equation}\label{orth cond newk}
\ps{b}{\frac{\partial^r \newk}{\partial x^r}(0,0)}{}=0, \quad \forall r \in \N,
\end{equation}
where $b \in \R^{n-1}$ is defined in \eqref{def b} and $\newk=(k_{ij})_{j \neq i}$.
It remains to relate the derivatives of the kernel at the origin with $M$.
This is the purpose of the next section.
This will be done only for a very well chosen particular solution to the kernel equations (i.e. for one $\delta \neq 1$) but this will be enough for our purposes as already underlined in Remark \ref{rem one choice is enough}.

\begin{remark}
We emphasize that, in all the sections below and unless specifically mentioned, the number of negative speeds $m$ is arbitrary (the orthogonality conditions \eqref{orth cond newk} are studied for any nonzero $b \in \R^{n-1}$).
\end{remark}

\section{The derivatives of the kernel at the origin}\label{sect deriv K}

\subsection{Normalization of the equations}\label{sect notaa}

Let us first observe that a feature of the kernel equations \eqref{kern equ}-\eqref{general BC} is that it does not couple different rows of $K$:
\begin{equation}\label{kern equ compo}
\begin{dcases}
\lambda_i\px{k_{ij}}(x,\xi)
+\pxi{k_{ij}}(x,\xi)\lambda_j
+\sum_{r=1}^n k_{ir}(x,\xi) m_{rj}=0,
\\
\lambda_i k_{ij}(x,x)-k_{ij}(x,x)\lambda_j=m_{ij} \quad (j \neq i),
\quad
k_{ii}(x,\delta x)=0.
\end{dcases}
\end{equation}
Therefore, all along Section \ref{sect deriv K}, $i \in \ens{1,\ldots,n}$ is fixed and we will drop the dependence on $i$ for clarity.

Let us now introduce some important notations.
Some of them have already been introduced in Section \ref{sect nota 1} but they are recalled here for the sake of the presentation.

\begin{itemize}
\item
It is convenient to normalize the kernel equations by $\lambda_i-\lambda_j$ for $j \neq i$ and by $\lambda_i$ otherwise.
The kernel equations \eqref{kern equ compo} become
\begin{equation}\label{kern equ j}
\begin{dcases}
\mu_j\px{k_j}(x,\xi)
+\pxi{k_j}(x,\xi)\nu_j
+\sum_{r=1}^n k_r(x,\xi) \newm_{r j}
=0,
\\
k_j(x,x)=\newm_{ij} \quad (j \neq i),
\quad
k_i(x,\delta x)=0,
\end{dcases}
\end{equation}
where $k=
\begin{pmatrix}
k_{i1} &
\cdots &
k_{in}
\end{pmatrix}^\tr$ and
$$
\mu_j=\frac{\lambda_i}{\lambda_i-\lambda_j},
\quad
\nu_j=\frac{\lambda_j}{\lambda_i-\lambda_j},
\quad (j \neq i),
\quad \mu_i=\nu_i=1,
$$

$$
\newm_{rj}=
\frac{m_{rj}}{\lambda_i-\lambda_j},
\quad (j \neq i),
\quad
\newm_{ri}=\frac{m_{ri}}{\lambda_i}.
$$
Note that, with this normalization, we have in particular $\mu_j-\nu_j=1$ for $j \neq i$.

\item
Since the component $k_i$ plays a different role than all the other components $k_j$ with $j \neq i$, we rewrite \eqref{kern equ j} in a matrix form separating both quantities.
Let us denote by $\newn=n-1$ and introduce $\newk=(k_j)_{j \neq i}$.
Then, system \eqref{kern equ j} can be written as
\begin{equation}\label{kern equ vec}
\begin{dcases}
D_\mu\px{\newk}(x,\xi)
+D_\nu \pxi{\newk}(x,\xi)
+A^\tr\newk(x,\xi)
+k_i(x,\xi) \seq^0
=0,
\\
\px{k_i}(x,\xi)
+\pxi{k_i}(x,\xi)
+\ps{w}{\newk(x,\xi)}{}
=0,
\\
\newk(x,x)=\seq^0,
\quad
k_i(x,\delta x)=0,
\end{dcases}
\end{equation}
where $D_\mu, D_\nu, A \in \R^{\newn \times \newn}$ are the matrices defined by
$$
D_\mu=\diag (\mu_j)_{j \neq i},
\quad
D_\nu=\diag (\nu_j)_{j \neq i},
\quad
A=(\newm_{rj})_{r,j \neq i},
$$
and $\seq^0, w \in \R^{\newn}$ are the vectors defined by
$$
\seq^0=(\newm_{ij})_{j \neq i},
\quad
w=(\newm_{ji})_{j \neq i}.
$$
Note that we used that $\newm_{ii}=0$ (since $M \in \Mset$).
Finally, it will also be convenient to use the matrix $D \in \R^{\newn \times \newn}$ and the vector $v \in \R^{\newn}$ defined by
\begin{equation}\label{def D and v}
D=D_\mu+D_\nu, \quad v=-\frac{1}{2}w.
\end{equation}

\end{itemize}

\subsection{Computation of the derivatives}

The main result of this section is the following.

\begin{theorem}\label{thm derivatives}
For the solution to \eqref{kern equ vec} with $\delta=-1$, we have
\begin{equation}\label{comp deriv newk}
\frac{\partial^{\gamma+\sigma} \newk}{\partial x^{\gamma} \partial \xi^{\sigma}}(0,0)
=\sum_{r=0}^{\gamma} \sum_{s=0}^{\sigma} (-1)^r \binom{\gamma}{r} \binom{\sigma}{s} \seq_{\gamma+\sigma-(r+s),r+s}, \quad \forall \gamma,\sigma \in \N,
\end{equation}
where $(\seq_{r,s})_{r,s \in \N}$ is the sequence defined by
\begin{equation}\label{def seq rs}
\seq_{r,0} =\seq^r,
\quad
\seq_{r,s} =0 \quad \text{ if } r<s,
\quad
\seq_{r,s} =
\sum_{\ell=0}^{r-s} \ps{v}{\seq_{r-1-\ell,s-1}}{} \seq_{\ell,0} \quad \text{ if } r \geq s \geq 1,
\end{equation}
where $\seqset$ is the sequence defined in \eqref{def seq}.
\end{theorem}

Combining this result with the estimates \eqref{estim k} and Taylor's theorem, we obtain an explicit formula for the solution to \eqref{kern equ vec} when $\delta=-1$:

\begin{corollary}\label{thm formula}
For $\delta=-1$, the solution to \eqref{kern equ vec} is given by
$$
\newk(x,\xi) = \sum_{r=0}^{+\infty}\sum_{s=0}^{+\infty} \frac{(-1)^r}{r ! s!} \seq_{r,s} (x-\xi)^r (x+\xi)^s,
\quad
k_i(x,\xi) = -
\int_{\frac{x-\xi}{2}}^x
\ps{w}{
\newk(\sigma,\sigma-x+\xi)
}{} \, d\sigma,
$$
for every $(x,\xi) \in \R^2$, where $(\seq_{r,s})_{r,s \in \N}$ is the sequence defined by \eqref{def seq rs} and the series is normally convergent on any compact set of $\R^2$.

\end{corollary}

\begin{remark}
Explicit solutions to the kernel equations were also obtained in \cite[Section 3.4]{VK14} for systems of $n=2$ equations.
\end{remark}

\begin{proof}[Proof of Theorem \ref{thm derivatives}]
\begin{enumerate}
\item
To explain the special role played by $\delta=-1$, we start the computations with an arbitrary $\delta \neq 1$.
The first idea is to form a system involving only $\newk$ by expressing $k_i$ as a function of $\newk$:
$$
k_i(x,\xi) =
-\int_{\frac{x-\xi}{1-\delta}}^x
\ps{w}{
\newk(\sigma,\sigma-x+\xi)
 }{} \, d\sigma.
$$
As a result, $\newk$ solves
$$
\begin{dcases}
D_\mu\px{\newk}(x,\xi)
+D_\nu \pxi{\newk}(x,\xi)
+A^\tr\newk(x,\xi)
-\left(\int_{\frac{x-\xi}{1-\delta}}^x
\ps{w}{ \newk(\sigma,\sigma-x+\xi) }{}
 \, d\sigma\right) \seq^0
=0,
\\
\newk(x,x)=\seq^0.
\end{dcases}
$$

We now transform this system into a Cauchy problem by introducing the transformation
$$
h(t,\theta)=\newk\left(\frac{-t+\theta}{2},\frac{t+\theta}{2}\right).
$$
Using that $D_\mu-D_\nu=\Id_{\newn}$, we can check that $h$ satisfies the system
\begin{equation}\label{equ U}
\begin{dcases}
\frac{\partial h}{\partial t}(t,\theta)
=D \frac{\partial h}{\partial \theta}(t,\theta)
+A^\tr h(t,\theta)
+\left(\int_{-\frac{1+\delta}{1-\delta}t}^{\theta}  \ps{v}{ h(t,\eta)}{} \, d\eta\right) \seq^0,
\\
h(0,\theta)=\seq^0,
\end{dcases}
\end{equation}
where we recall that $D$ and $v$ are defined in \eqref{def D and v}.
Note as well that
$$
\frac{\partial^{\gamma+\sigma} \newk}{\partial x^{\gamma} \partial \xi^{\sigma}}(x,\xi)
=\sum_{r=0}^{\gamma} \sum_{s=0}^{\sigma} (-1)^r \binom{\gamma}{r} \binom{\sigma}{s} 
\frac{\partial^{\gamma+\sigma} h}{\partial t^{\gamma+\sigma-(r+s)} \partial \theta^{r+s}}(-x+\xi,x+\xi)
, \quad \forall \gamma,\sigma \in \N,
$$
so that the derivatives of $\newk$ at $(0,0)$ can be deduced from those of $h$.
They will be computed from \eqref{equ U} and we see that the computations considerably simplify if the lower bound of the integral vanishes, that is if we choose $\delta=-1$.
For this choice, we define
$$
\seq_{r,s}=\frac{\partial^{r+s} h}{\partial t^r \partial \theta^s}(0,0).
$$
We are going to show that it satisfies \eqref{def seq rs}.

\item
All along the rest of the proof, we will use the notation $c_{r,s}=\ps{v}{\seq_{r,s}}{}$.
First observe that system \eqref{equ U} (with $\delta=-1$) yields the following identities:
\begin{equation}\label{equ X}
\begin{dcases}
\seq_{r+1,0}=D\seq_{r,1}+A^\tr\seq_{r,0}, \\
\seq_{r+1,s}=D\seq_{r,s+1}+A^\tr\seq_{r,s}+ c_{r,s-1} \seq_{0,0}, \\
\seq_{0,s}=0,
\end{dcases}
\end{equation}
for every $r \geq 0$ and $s \geq 1$.
The second property in \eqref{def seq rs} is easily proved by induction on $r \geq 1$.
To establish the two other identities, it is sufficient to prove the following statement for any $N \geq 1$:
\begin{equation}\label{property N}
\begin{dcases}
\seq_{s+q,s}=
\sum_{\ell=0}^q c_{s-1+q-\ell,s-1} \seq_{\ell,0},
\quad \forall s \geq 1, \, \forall 0 \leq q \leq N,
\\
\seq_{r,0}=
A^\tr\seq_{r-1,0}
+D \sum_{\ell=0}^{r-2} c_{r-2-\ell,0} \seq_{\ell,0},
\quad \forall 2 \leq r \leq N+1.
\end{dcases}
\end{equation}

We prove it by induction.
For $N=1$, this is clear.
Indeed, for any $s \geq 1$, we have
\begin{align}
\seq_{s,s} &=D\seq_{s-1,s+1}+A^\tr\seq_{s-1,s}+c_{s-1,s-1}\seq_{0,0} \quad \text{(by \eqref{equ X})},
\nonumber
\\&
=c_{s-1,s-1} \seq_{0,0} \quad \text{(by the second property in \eqref{def seq rs})},
\label{id Xss}
\end{align}

\begin{align*}
\seq_{s+1,s} &=D\seq_{s,s+1}+A^\tr\seq_{s,s}+c_{s,s-1}\seq_{0,0} \quad \text{(by \eqref{equ X})},
\\&
=A^\tr\seq_{s,s} +c_{s,s-1}\seq_{0,0} \quad \text{(by the second property in \eqref{def seq rs})},
\\&
=c_{s-1,s-1} \seq_{1,0} +c_{s,s-1} \seq_{0,0}
\quad \text{(by \eqref{id Xss})},
\end{align*}

and
\begin{align*}
\seq_{2,0} &=D\seq_{1,1}+A^\tr\seq_{1,0}
\quad \text{(by \eqref{equ X})},
\\&
=c_{0,0} D\seq_{0,0}+A^\tr\seq_{1,0}
 \quad \text{(by \eqref{id Xss})}.
\end{align*}

Assume now that \eqref{property N} holds for $N \geq 1$ and let us prove it for $N+1$.
We first show that
\begin{equation}\label{equ q for N+1}
\seq_{s+N+1,s}=
\sum_{\ell=0}^{N+1} c_{s+N-\ell,s-1} \seq_{\ell,0}, \quad \forall s \geq 1.
\end{equation}

For any $s \geq 1$, we have
\begin{align*}
\seq_{s+N+1,s} &=
D\seq_{s+N,s+1} +A^\tr\seq_{s+N,s} +c_{s+N,s-1} \seq_{0,0}
\quad \text{(by \eqref{equ X})},
\\&
=D\sum_{\ell=0}^{N-1} c_{s+N-1-\ell,s} \seq_{\ell,0}
+A^\tr\sum_{\ell=0}^N c_{s-1+N-\ell,s-1} \seq_{\ell,0}
\\&
+c_{s+N,s-1}\seq_{0,0}
\quad \text{(by assumption \eqref{property N})},
\\&
=D\sum_{\ell=0}^{N-1} c_{s+N-1-\ell,s} \seq_{\ell,0}
+\sum_{r=2}^{N+1} c_{s+N-r,s-1} A^\tr \seq_{r-1,0}
\\&
+c_{s-1+N,s-1} \seq_{1,0}
+c_{s+N,s-1} \seq_{0,0}.
\end{align*}

Therefore, if we show the identity
\begin{equation}\label{remaing id}
\sum_{\ell=0}^{N-1} c_{s+N-1-\ell,s} \seq_{\ell,0}=
\sum_{r=2}^{N+1} c_{s+N-r,s-1}
\left(
\sum_{\ell=0}^{r-2} c_{r-2-\ell,0} \seq_{\ell,0}
\right),
\end{equation}
then we can use the second condition in \eqref{property N} to obtain the desired identity \eqref{equ q for N+1}.
To establish \eqref{remaing id}, we use the first condition in \eqref{property N} to deduce that, for every $0 \leq \ell \leq N-1$,
\begin{align*}
\ps{v}{\seq_{s+N-1-\ell,s}}{}
&=
\sum_{j=0}^{N-1-\ell} c_{s-1+N-1-\ell-j,s-1} c_{j,0}
\\&=
\sum_{r=\ell+2}^{N+1} c_{s+N-r,s-1} c_{r-2-\ell,0}.
\end{align*}
Finally, a simple change of order of summation leads to \eqref{remaing id}.
It remains to show the second identity in \eqref{property N} for $r=N+2$, namely
$$
\seq_{N+2,0}=
A^\tr\seq_{N+1,0}
+D \sum_{\ell=0}^N c_{N-\ell,0} \seq_{\ell,0}.
$$
We have
\begin{align*}
\seq_{N+2,0} &=D\seq_{N+1,1}+A^\tr\seq_{N+1,0}
\quad \text{(by \eqref{equ X})},
\\&
=D\sum_{\ell=0}^N c_{N-\ell,0} \seq_{\ell,0}
+A^\tr\seq_{N+1,0}
\quad \text{(by \eqref{property N})}.
\end{align*}

\end{enumerate}

\end{proof}

\begin{remark}\label{rem formula other BC}
Theoretically, we can also compute all the derivatives at $(0,0)$ of the solution to \eqref{kern equ vec} for arbitrary $\delta \neq 1$.
This can be done by taking derivatives and inverting some matrix.
However, the size of this matrix grows with the order of derivatives and computations rapidly become more and more complicated.
It seems difficult with such a strategy to obtain a suitable formula for arbitrary $\delta$.
At the same time, we see from \eqref{equ U} that a different choice of $\delta$ means more derivatives to be computed, as for instance with $\delta=0$ which leads to an integral of the form $\int_{-t}^{\theta}$.
In addition to that, we recall that one choice of $\delta$ is actually sufficient for the purposes of this paper (Remark \ref{rem one choice is enough}).
\end{remark}

\section{Study of the orthogonality conditions}\label{sect hyperplane}

In this section, we use the computations obtained in the previous section to study the orthogonality conditions \eqref{orth cond newk}.
We start with the conclusion of the proof of our first main result.

\begin{proof}[Proof of Theorem \ref{thm M}]
We recall that, from the results of the previous sections, we only have to show the equivalence between the orthogonality conditions \eqref{orth cond newk} and \eqref{cond seq}.
First observe that, from the definition \eqref{def seq rs} of the sequence $(\seq_{r,s})_{r,s \in \N}$, it is clear that \eqref{cond seq} is equivalent to
$$\ps{b}{\seq_{r,s}}{}=0, \quad \forall r,s \in \N.$$
We can check that this condition is equivalent to \eqref{orth cond newk} using \eqref{comp deriv newk} and \eqref{def seq rs}.
\end{proof}

\begin{remark}\label{rem equiv char}
The proof above and the analyticity of $\newk$ in $\R^2$ shows that the following four properties are in fact equivalent:
\begin{enumerate}
\item\label{prop equiv char i1}
$\ps{b}{\newk(x,\xi)}{}=0$ for every $(x,\xi) \in \R^2$.

\item\label{prop equiv char i2}
$\ps{b}{\newk(x,0)}{}=0$ for every $x \in \R$.

\item\label{prop equiv char i3}
$\ps{b}{\seq_{r,0}}{}=0$ for every $r \in \N$.

\item\label{prop equiv char i4}
$\ps{b}{\seq_{r,s}}{}=0$ for every $r,s \in \N$.
\end{enumerate}
\end{remark}

We are now going to study the orthogonality conditions \eqref{cond seq} for the sequence $\seqset$ and prove our second main result.
For the rest of Section \ref{sect hyperplane}, the matrices $A,D \in \R^{\newn \times \newn}$ and the vectors $\seq^0, v, b \in \R^{\newn}$ can in fact be arbitrary.
We emphasize that $n$ is also arbitrary, it is only during the proof of Theorem \ref{thm bounds Nseq} that we will assume that $n=3$ or $n=4$.

\subsection{Some invariant subspaces of the sequence}\label{sect inv sub}

We start the general study of the orthogonality conditions \eqref{cond seq} with the description of two simple invariant subspaces of the sequence $\seqset$.

\begin{proposition}\label{prop inv sub}
Assume that $\seq^0 \in E$ for some $E \subset \R^{\newn}$ satisfying one of the following two conditions:
\begin{align}
& A^\tr(E) \subset E, \quad E \subset \ker v^\tr.
\label{first E}
\\
& A^\tr(E) \subset E, \quad D(E) \subset E. \label{second E}
\end{align}
Then, $\seq^r \in E$ for every $r \in \N$.
\end{proposition}

\begin{proof}
We prove the result by induction on $r$.
For $r=0$ this is trivial and for $r=1$ this follows from the definition $\seq^1=A^\tr\seq^0$ and the property $A^\tr(E) \subset E$.
Assume then that $\seq^\ell \in E$ for every $0 \leq \ell \leq r$ for some $r \geq 1$ and let us show that $\seq^{r+1} \in E$.
Since $r+1 \geq 2$, we have
$$
\seq^{r+1}=
A^\tr\seq^r
+\sum_{\ell=0}^{r-1} \ps{v}{ \seq^{r-1-\ell}}{} D\seq^{\ell}.
$$

Clearly, the first part $A^\tr\seq^r$ belongs to $E$ since $\seq^r \in E$ and $A^\tr(E) \subset E$.
The remaining part also belongs to $E$ since either $\ps{v}{\seq^{\ell}}{}=0$ for every $0 \leq \ell \leq r$ (if $E \subset \ker v^\tr $) or $D\seq^{\ell} \in E$ for every $0 \leq \ell \leq r$ (if $D(E) \subset E$).
\end{proof}

If we can find a subspace $E$ such as in the previous proposition and which is in addition included in $\ker b^\tr$, then we see that the whole sequence will be guaranteed to stay in $\ker b^\tr$.

\subsection{Characterization of rank one sequences}\label{sect rank 1}

In this section, we characterize when the rank of $\seqset$ is equal to one, and we use it to prove Theorem \ref{thm bounds Nseq} in the case $n=3$.
We recall that, by definition, $\rank (\seq^r)_{r \in S}=\dim \Span \ens{\seq^r \st r \in S}$ for any $S \subset \N$.

From now on, it will be convenient to use the following notation:
$$
E_s=\Span\ens{\seq^r \st r \leq s}, \quad \forall s \in \N.
$$

First of all, it is clear that
$$
\rank \seqset=1
\quad \Longleftrightarrow \quad
\left(\seq^0 \neq 0, \quad \seq^r \in \Ezero, \quad \forall r \geq 1\right).
$$
We have the following characterization:

\begin{proposition}\label{thm caract E1}
The following three conditions are equivalent:
\begin{enumerate}
\item\label{caract E1 i1}
$\seq^r \in \Ezero$ for every $r \geq 1$.

\item\label{caract E1 i2}
$\seq^r \in \Ezero$ for every $r \in \ens{1,2}$.

\item\label{caract E1 i3}
$\Ezero$ satisfies \eqref{first E} or \eqref{second E}.
\end{enumerate}

\end{proposition}

\begin{proof}
The implication \ref{caract E1 i1} $\Longrightarrow$ \ref{caract E1 i2} is trivial.
The implication \ref{caract E1 i3} $\Longrightarrow$ \ref{caract E1 i1} follows from Proposition \ref{prop inv sub}.
Let us show that \ref{caract E1 i2} $\Longrightarrow$ \ref{caract E1 i3}.
We write
$$
\seq^r=\alpha_r \seq^0, \quad r=1,2,
$$
for some $\alpha_r \in \R$.
The condition for $r=1$ gives $A^\tr\seq^0=\alpha_1 \seq^0$, which is equivalent to $A^\tr(\Ezero) \subset \Ezero$.
The condition for $r=2$ gives $A^\tr\seq^1 + \ps{v}{\seq^0}{} D\seq^0=\alpha_2 \seq^0$, which implies $\ps{v}{\seq^0}{} D\seq^0 \in \Ezero$, that is either $\Ezero \subset \ker v^\tr$ or $D(\Ezero) \subset \Ezero$.
This establishes the desired equivalences.

\end{proof}

\begin{proof}[Proof of Theorem \ref{thm bounds Nseq} (case $n=3$)]
Assume that
$$
\ps{b}{\seq^r}{}=0, \quad \forall r \in \ens{0,1,2}.
$$
Since $\newn=2$, we have $\dim \ker b^\tr =1$ and thus
$$\rank (\seq^r)_{r \in \ens{0,1,2}} \leq 1.$$
As a result, $\seq^r \in \Ezero$ for $r \in \ens{1,2}$ and it follows that $\seq^r \in \Ezero$ for every $r \geq 1$ by Proposition \ref{thm caract E1}.
\end{proof}

\subsection{Characterization of rank two sequences}\label{sect rank 2}

In this section, we characterize when the rank of $\seqset$ is equal to two, and we use it to prove Theorem \ref{thm bounds Nseq} in the case $n=4$.
The study of rank two sequences is considerably more difficult than for rank one.

We start with the following simple observation.

\begin{proposition}\label{prop rank 2 basic}
Let $\ens{0,1,2} \subset S \subset \N$.
We have $\rank (\seq^r)_{r \in S}=2$ if, and only if, we have one of the following two conditions:
\begin{enumerate}
\item\label{prop rank 2 i1}
$\rank(\seq^0|\seq^1)=1$, $\rank(\seq^0|\seq^2)=2$ and $\seq^r \in \Etwo$ for every $r \in S \setminus \ens{0,1,2}$.

\item\label{prop rank 2 i2}
$\rank(\seq^0|\seq^1)=2$ and $\seq^r \in \Eone$ for every $r \in S \setminus \ens{0,1}$.
\end{enumerate}

\end{proposition}

\begin{proof}
We only need to observe that the situation $\rank(\seq^0|\seq^1)=\rank(\seq^0|\seq^2)=1$ does not happen.
Indeed, in that situation we have $\seq^1,\seq^2 \in \Ezero$ and thus $\seq^r \in \Ezero$ for every $r \geq 1$ by Proposition \ref{thm caract E1}, which shows that the sequence cannot be of rank two.
\end{proof}

We now characterize both conditions of Proposition \ref{prop rank 2 basic}.

\begin{proposition}\label{prop rank 2 basic i1}
Assume that $\rank(\seq^0|\seq^1)=1$ and $\rank(\seq^0|\seq^2)=2$.
Then, the following three conditions are equivalent:
\begin{enumerate}
\item\label{prop rank 2 basic i1 i1}
$\seq^r \in \Etwo$ for every $r \geq 3$.

\item\label{prop rank 2 basic i1 i2}
$\seq^r \in \Etwo$ for every $r \in \ens{3,4}$.

\item\label{prop rank 2 basic i1 i3}
$\Etwo$ satisfies \eqref{second E}.
\end{enumerate}

\end{proposition}

\begin{proof}
The implication \ref{prop rank 2 basic i1 i1} $\Longrightarrow$ \ref{prop rank 2 basic i1 i2} is trivial.
The implication \ref{prop rank 2 basic i1 i3} $\Longrightarrow$ \ref{prop rank 2 basic i1 i1} follows from Proposition \ref{prop inv sub}.
Let us show that \ref{prop rank 2 basic i1 i2} $\Longrightarrow$ \ref{prop rank 2 basic i1 i3}.
From the definition of $\seq^2$ and the rank assumptions, we have $\ps{v}{\seq^0}{} \neq 0$ and $D\seq^0 \in \Span\ens{\seq^0,\seq^2}$.
Using these facts, we can first check that the condition $\seq^3 \in \Etwo$ gives $A^\tr(\Etwo) \subset \Etwo$ and then that the condition $\seq^4 \in \Etwo$ yields $D(\Etwo) \subset \Etwo$.
\end{proof}

The second condition in Proposition \ref{prop rank 2 basic} is more difficult to characterize.

\begin{proposition}\label{thm caract E2}
Assume that $\rank(\seq^0|\seq^1)=2$.
Then, the following two conditions are equivalent:
\begin{enumerate}
\item\label{caract E2 i1}
$\seq^r \in \Eone$ for every $r \geq 2$.

\item\label{caract E2 i2}
$\seq^r \in \Eone$ for every $r \in \ens{2,\ldots, 5}$.
\end{enumerate}

\end{proposition}

Before proving Proposition \ref{thm caract E2}, we prove our second main result.

\begin{proof}[Proof of Theorem \ref{thm bounds Nseq} (case $n=4$)]
Assume that
$$
\ps{b}{\seq^r}{}=0, \quad \forall r \in \ens{0,\ldots,5}.
$$
Since $\newn=3$, we have $\dim \ker b^\tr =2$ and thus
$$\rank (\seq^r)_{r \in \ens{0,\ldots,5}} \leq 2.$$
If the rank is in fact less than or equal to $1$, then we conclude as in the proof of Theorem \ref{thm bounds Nseq} in the case $n=3$.
If the rank is exactly equal to $2$,
then the conclusion follows from Propositions \ref{prop rank 2 basic}, \ref{prop rank 2 basic i1} and \ref{thm caract E2}.
\end{proof}

We now turn to the proof of the key proposition.

\begin{proof}[Proof of Proposition \ref{thm caract E2}]
If $\Eone$ satisfies \eqref{first E} or \eqref{second E}, then the result follows from Proposition \ref{prop inv sub}.
Therefore, from now on, we assume that $E_1$ does not meet any of these conditions.

\begin{enumerate}
\item
Let $N \geq 4$ be arbitrary and consider the property: for every $r \in \ens{2,\ldots,N}$, there exist $\alpha_r,\beta_r \in \R$ such that
\begin{equation}\label{decomp psik}
\seq^r=\alpha_r \seq^0 + \beta_r \seq^1.
\end{equation}
Let us also set
\begin{equation}\label{def seq 01}
\begin{aligned}
& \alpha_0=1, \quad \beta_0=0, \\
& \alpha_1=0, \quad \beta_1=1,
\end{aligned}
\end{equation}
so that the previous identity is always true for $r=0,1$.
Using the definition \eqref{def seq} of the sequence we see that, for $r \geq 2$, identity \eqref{decomp psik} is equivalent to
\begin{align*}
\alpha_r \seq^0 + \beta_r \seq^1
=&
\alpha_{r-1} A^\tr\seq^0 + \beta_{r-1} A^\tr\seq^1
\\&
+\left(\sum_{\ell=0}^{r-2} c_{r-2-\ell} \alpha_{\ell}\right) D\seq^0
+\left(\sum_{\ell=0}^{r-2} c_{r-2-\ell} \beta_{\ell}\right) D\seq^1,
\end{align*}
where we introduced, for every $s \in \ens{0,\ldots,N}$, 
$$c_s=\alpha_s \ps{v}{\seq^0}{} + \beta_s \ps{v}{\seq^1}{}.$$

Let us eliminate the terms $A^\tr\seq^0, A^\tr\seq^1$.
By definition of the sequence, we have $A^\tr\seq^0=\seq^1$.
On the other hand, condition \eqref{decomp psik} for $r=2$ yields
\begin{equation}\label{E2 cond r2}
A^\tr\seq^1=\alpha_2 \seq^0 +\beta_2 \seq^1 -c_0D\seq^0.
\end{equation}

As a result, for $r \geq 3$, identity \eqref{decomp psik} is equivalent to
\begin{equation}\label{cond any r}
(\seq^0| \seq^1) u_r
=(D\seq^0| D\seq^1) f_{r-1}
,
\end{equation}
where $u_r, f_{r-1} \in \R^2$ are the vectors defined by
\begin{equation}\label{def u f}
u_r=
\begin{pmatrix}
\alpha_r
-\beta_{r-1}\alpha_2
\\
\beta_r
-\left(\alpha_{r-1} +\beta_{r-1} \beta_2\right)
\end{pmatrix},
\quad
f_{r-1}=
\begin{pmatrix}
\beta_{r-1}(-c_0)
+\sum_{\ell=0}^{r-2} c_{r-2-\ell} \alpha_{\ell}
\\
\sum_{\ell=0}^{r-2} c_{r-2-\ell} \beta_{\ell}
\end{pmatrix}.
\end{equation}

Let us make some observations.
\begin{itemize}
\item
For any $r_1,r_2$, if $\lambda f_{r_1-1}+ \mu f_{r_2-1}=0$ for some $\lambda,\mu \in \R$, then $\lambda u_{r_1}+ \mu u_{r_2}=0$.
Indeed, denoting by $u=\lambda u_{r_1}+ \mu u_{r_2}$, we have $(\seq^0| \seq^1) u=0$ and thus $u=0$ since we assumed that $\seq^0,\seq^1$ are linearly independent.

\item
$f_{r_1-1}$ and $f_{r_2-1}$ are linearly dependent for any $r_1,r_2$.
Indeed, otherwise we obtain
$(\seq^0| \seq^1) U=(D\seq^0| D\seq^1)$ for some matrix $U \in \R^{2\times2}$.
This means that $D(\Eone) \subset \Eone$.
It then follows from \eqref{E2 cond r2} that $A^\tr(\Eone) \subset \Eone$ as well.
Therefore, $E_1$ satisfies \eqref{second E}, but this situation has been excluded from the beginning of the proof.

\item
$f_2$ and $f_3$ are linearly dependent only if $c_0 \neq 0$.
Indeed, if $c_0=0$, then $A^\tr\seq^1 \in \Eone$ by \eqref{E2 cond r2}, so that $A^\tr(\Eone) \subset \Eone$.
If $c_0=0$, we also have $\det(f_2|f_3)=c_1^2$, which cannot be zero since otherwise $c_0=c_1=0$, that is $\Eone \subset \ker v^\tr $, and thus $E_1$ satisfies \eqref{first E} (excluded).

\end{itemize}

The necessary condition $\det(f_2|f_{r-1})=0$ is equivalent to the identity
$$
\left(\sum_{\ell=0}^{r-2} c_{r-2-\ell} \beta_{\ell}\right) f_2 -c_0 f_{r-1}=0,$$
which, by the first observation above, in turn implies that
$$
\left(\sum_{\ell=0}^{r-2} c_{r-2-\ell} \beta_{\ell}\right) u_3 -c_0 u_r=0.$$
Since $c_0 \neq 0$, this gives the following formulas:
\begin{equation}\label{formulas f u}
\begin{dcases}
f_{r-1}=\left(\sum_{\ell=0}^{r-2} \newc_{r-2-\ell} \beta_{\ell}\right) f_2, \\
u_r= \left(\sum_{\ell=0}^{r-2} \newc_{r-2-\ell} \beta_{\ell}\right) u_3,
\end{dcases}
\end{equation}
where we introduced $\newc_s=\frac{c_s}{c_0}$.
Conversely, it is clear that if we have \eqref{formulas f u}, then \eqref{cond any r} will also hold for $r \geq 3$, provided that it holds for $r=3$.

Finally, observe that the second formula in \eqref{formulas f u}, combined with the definition \eqref{def u f} of $u_r$, uniquely determines all the $\alpha_r,\beta_r$ for $r \geq 4$ as a function of $\alpha_2,\beta_2$ and $\alpha_3,\beta_3$:
\begin{equation}\label{def seq ur}
\begin{dcases}
\alpha_r = \beta_{r-1}\alpha_2
+\left(\sum_{\ell=0}^{r-2} \newc_{r-2-\ell} \beta_{\ell}\right)
\left(\alpha_3 -\beta_2\alpha_2\right),
\\
\beta_r=\alpha_{r-1} +\beta_{r-1} \beta_2
+\left(\sum_{\ell=0}^{r-2} \newc_{r-2-\ell} \beta_{\ell}\right)
\left(\beta_3 -(\alpha_2 +\beta_2^2)\right),
\end{dcases}
\end{equation}
and that the first formula in \eqref{formulas f u} is equivalent to
\begin{equation}\label{cond beta}
\beta_{r-1} =\sum_{\ell=0}^{r-2} \newc_{r-2-\ell} \left(\alpha_\ell + (\beta_2-\newc_1)\beta_\ell\right).
\end{equation}

In summary, we have shown that the property considered is equivalent to: $c_0 \neq 0$ and there exist $\alpha_2, \beta_2$ and $\alpha_3, \beta_3$ such that \eqref{decomp psik} holds for $r=2,3$ and such that the sequence defined by \eqref{def seq ur} (with \eqref{def seq 01}) satisfies \eqref{cond beta} for every $r \in \ens{4,\ldots,N}$.

\item
Let us now study the sequence \eqref{def seq ur}.
The proof of the result will be complete after we show that the following three conditions are equivalent:
\begin{enumerate}
\item\label{cond beta i1}
Condition \eqref{cond beta} holds for every $r \geq 4$.

\item\label{cond beta i2}
Condition \eqref{cond beta} holds for $r=4,5$.

\item\label{cond beta i3}
$\alpha_3$ and $\beta_3$ are given by
\begin{equation}\label{cond a3b3}
\begin{dcases}
\beta_3 =
(\alpha_2+\beta_2^2)
-(\beta_2-\newc_1)^2
+\alpha_2 +(\beta_2-\newc_1)\beta_2,
\\
\alpha_3 =
\beta_2\alpha_2 - (\beta_2 -\newc_1) (\beta_3-(\alpha_2+\beta_2^2)).
\end{dcases}
\end{equation}

\end{enumerate}

We start with the implication \ref{cond beta i2} $\Longrightarrow$ \ref{cond beta i3}.
Condition \eqref{cond beta} for $r=4$ immediately gives $\beta_3$ as a function of $\alpha_2, \beta_2$:
$$
\beta_3 =\sum_{\ell=0}^2 \newc_{2-\ell} \left(\alpha_\ell + (\beta_2-\newc_1)\beta_\ell\right).
$$

On the other hand, condition \eqref{cond beta} for $r=5$ gives
\begin{align*}
\beta_4 &=\sum_{\ell=1}^2 \newc_{3-\ell} \left(\alpha_\ell + (\beta_2-\newc_1)\beta_\ell\right)
+\newc_0 \left(\alpha_3 + (\beta_2-\newc_1)\beta_3\right)
+\newc_3 \left(\alpha_0 + (\beta_2-\newc_1)\beta_0\right)
\\
 &=2\alpha_3 + \beta_2 \beta_3
+\sum_{\ell=1}^2 \newc_{3-\ell} \left(\alpha_\ell + (\beta_2-\newc_1)\beta_\ell\right),
\end{align*}
whereas, by definition \eqref{def seq ur},
$$
\beta_4=\alpha_3 +\beta_3 \beta_2
+\left(\sum_{\ell=0}^2 \newc_{2-\ell} \beta_{\ell}\right)
\left(\beta_3 -(\alpha_2 +\beta_2^2)\right).
$$

Identifying both expressions gives $\alpha_3$ as a function of $\alpha_2, \beta_2$:
$$
\alpha_3
=
-\sum_{\ell=1}^2 \newc_{3-\ell} \left(\alpha_\ell + (\beta_2-\newc_1)\beta_\ell\right)
+\left(\sum_{\ell=0}^2 \newc_{2-\ell} \beta_{\ell}\right)
\left(\beta_3 -(\alpha_2 +\beta_2^2)\right).
$$

We can check that the previous formulas are equivalent to \eqref{cond a3b3} (we prefer the expressions in \eqref{cond a3b3} because they make appear some coefficients involved in \eqref{def seq ur}).

Let us now prove the implication \ref{cond beta i3} $\Longrightarrow$ \ref{cond beta i1}.
We prove it by induction on $r$.
For $r=4$,  this holds by very definition of $\beta_3$ as we have seen above.
Assume now that \eqref{cond beta} holds for some arbitrary $r \geq 4$, and let us prove it for $r+1$, that is to prove that we have
\begin{equation}\label{formula betar induc}
\beta_r =\sum_{\ell=0}^{r-1} \newc_{r-1-\ell} \left(\alpha_\ell + (\beta_2-\newc_1)\beta_\ell\right).
\end{equation}

By definition \eqref{def seq ur} of $\beta_r$, we have
\begin{align*}
\beta_r =&
\alpha_{r-1}
+\beta_{r-1} \newc_1
+(\beta_2-\newc_1)\left(
\beta_{r-1}
-(\beta_2-\newc_1)\sum_{\ell=0}^{r-2} \newc_{r-2-\ell} \beta_{\ell}
\right)
\\&
+\left(
\beta_3 -(\alpha_2 +\beta_2^2)
+(\beta_2-\newc_1)^2
\right)
\sum_{\ell=0}^{r-2} \newc_{r-2-\ell} \beta_{\ell}.
\end{align*}

Using the induction hypothesis \eqref{cond beta} and the definition \eqref{cond a3b3} of $\beta_3$, we obtain
$$
\beta_r =
\newc_{r-1}
+ \sum_{\ell=0}^{r-2} \newc_{r-2-\ell} \left(
(\beta_2-\newc_1) \alpha_{\ell}
+(\alpha_2 +(\beta_2-\newc_1)\beta_2) \beta_{\ell}
\right).
$$

Now observe that, using the definition of $\alpha_3$, we have
\begin{align*}
\alpha_{\ell+1} + (\beta_2-\newc_1)\beta_{\ell+1}
&=\beta_\ell \alpha_2 +(\beta_2-\newc_1)(\alpha_\ell +\beta_\ell\beta_2)
\\
&=(\beta_2-\newc_1) \alpha_\ell +(\alpha_2 +(\beta_2-\newc_1)\beta_2) \beta_\ell.
\end{align*}
Formula \eqref{formula betar induc} easily follows this identity and the previous one.

\end{enumerate}

\end{proof}

\section{Kernel associated with the orthogonality conditions}\label{sect other explicit kern}

In this section, we supplement our results by giving a more explicit characterization of the conditions found in the previous section and that guaranteed the orthogonality conditions.
Then, we discuss the structure of the associated kernel.

\subsection{Kernel associated with the invariant subspaces}

Here we discuss properties related to the invariant subspaces of Section \ref{sect inv sub}.
We recall that, for any vector $h \in \R^{\newn}$, we denote the Kalman matrix of $(A,h)$ by
$$\kal{A,h}=(h|Ah|A^2h|\cdots|A^{\newn-1}h) \in \R^{\newn \times \newn}.$$

\begin{proposition}
\begin{enumerate}
\item
There exists $E$ satisfying \eqref{first E}, $\seq^0 \in E$ and $E \subset \ker b^\tr$, if, and only if,
\begin{equation}\label{cond hyper 1}
\kal{A,v}^\tr  \seq^0=\kal{A,b}^\tr  \seq^0=0.
\end{equation}

\item
Assume that $\kal{A,v}^\tr  \seq^0=0$.
Then, for any $\delta \neq 1$, the solution to the kernel equations \eqref{kern equ vec} is
$$
k_i(x,\xi)=0, \quad \newk(x,\xi)=e^{-A^\tr(x-\xi)} \seq^0, \quad \forall (x,\xi) \in \R^2.
$$
If, moreover, $\kal{A,b}^\tr  \seq^0=0$, then the orthogonality conditions \eqref{orth cond newk} are satisfied.
\end{enumerate}

\end{proposition}

\begin{proof}
\begin{enumerate}
\item
Assume that \eqref{cond hyper 1} holds.
Let us define
$$E=\ker \kal{A,v}^\tr \cap \ker \kal{A,b}^\tr.$$
By assumption, $\seq^0 \in E$ and it is clear that $E \subset \ker v^\tr$ and $E \subset \ker b^\tr$.
Finally, $E$ is stable by $A^\tr$ thanks to Cayley-Hamilton theorem.
Conversely, assume that \eqref{first E} holds for some $E \subset \ker b^\tr$ with $\seq^0 \in E$.
Since $\seq^0 \in E$ and $E$ is stable by $A^\tr$, we have $(A^\tr)^k \seq^0 \in E$ for every $k \in \N$.
Since $E \subset \ker v^\tr$ and $E \subset \ker b^\tr$, we obtain \eqref{cond hyper 1}.

\item
We see from the kernel equations \eqref{kern equ vec} that $k_i=0$ if, and only if, we have
$$
\begin{dcases}
D_\mu\px{\newk}(x,\xi)
+D_\nu \pxi{\newk}(x,\xi)
+A^\tr \newk(x,\xi)
=0,
\\
\ps{v}{\newk(x,\xi)}{}=0,
\\
\newk(x,x)=\seq^0.
\end{dcases}
$$
Using that $D_\mu-D_\nu=\Id_{\newn}$, it is clear that $\newk(x,\xi)=e^{-A^\tr(x-\xi)} \seq^0$ satisfies the first equation.
The second condition follows from the assumption $\kal{A,v}^\tr \seq^0=0$ and Cayley-Hamilton theorem.
The third condition is trivial.
Finally, the orthogonality conditions are clearly satisfied under the additional assumption $\kal{A,b}^\tr  \seq^0=0$.

\end{enumerate}

\end{proof}

Let us now address the second type of invariant subspaces introduced in Section \ref{sect inv sub}.

\begin{proposition}\label{prop caract inv sub}
Assume that $\seq^0, b \neq 0$.
\begin{enumerate}
\item
There exists $E$ satisfying \eqref{second E}, $\seq^0 \in E$ and $E \subset \ker b^\tr$ if, and only if, there exists a nonempty $J \subsetneq \ens{1,\ldots,\newn}$ such that
$$
\seq^0_j=a_{rj}=b_r=0, \quad \forall j \not\in J, \, \forall r \in J.
$$

\item
Assume that there exists a nonempty $J \subsetneq \ens{1,\ldots,\newn}$ such that $\seq^0_j=a_{rj}=0$ for every $j \not\in J$ and $r \in J$.
Then, for any $\delta \neq 1$, the solution to the kernel equations \eqref{kern equ vec} satisfies
\begin{equation}\label{cond newkj zero}
\newk_{j}=0, \quad \forall j \not\in J.
\end{equation}
If, moreover, $b_r=0$ for every $r \in J$, then the orthogonality conditions \eqref{orth cond newk} are satisfied.
\end{enumerate}

\end{proposition}

\begin{proof}
\begin{enumerate}
\item
Since $D$ is a diagonal matrix with distinct entries, its invariant subspaces are of the form
$$E=\Span\ens{e_r \st r \in J},$$
for some $J \subset \ens{1,\ldots,\newn}$, where $e_1,\ldots,e_{\newn}$ are the canonical vectors of $\R^{\newn}$.
Since $\seq^0 \neq 0$ (resp. $b \neq 0$), we have $J \neq \emptyset$ (resp. $J \neq \ens{1,\ldots,\newn}$).
Then, we easily check that such a subspace is invariant by $A^\tr$ if, and only if, $a_{rj}=0$ for every $r \in J$ and $j \not\in J$ and that it is included in $\ker b^\tr$ if, and only if, $b_r=0$ for every $j \in J$.

\item
Property \eqref{cond newkj zero} is a consequence of the uniqueness of the solution to the kernel equations.
The orthogonality conditions are clearly satisfied under the additional assumption that $b_r=0$ for every $r \in J$.

\end{enumerate}

\end{proof}

\begin{remark}
The first item in the above propositions gives explicit conditions that guarantee that the orthogonality conditions \eqref{orth cond newk} hold (when combined with the results of the previous sections).
We found these conditions with an algebraic approach.
On the other hand, once these conditions are known, the second item of the above propositions show how to use them to obtain an analytic proof of the orthogonality conditions.
Observe in addition that these different proofs are valid for arbitrary $\delta \neq 1$.
\end{remark}

\subsection{Kernel associated with nontrivial rank two sequences}

In the same spirit as in the previous section, we now we discuss the following property, related to Proposition \ref{thm caract E2}:
\begin{equation}\label{cond Eone r to five}
E_1 \subset \ker b^\tr, \quad \seq^r \in \Eone, \quad \forall r \in \ens{2,\ldots, 5}.
\end{equation}

Below, we denote by $c_0=\ps{v}{\seq^0}{}$.

\begin{proposition}\label{prop caract rank two seq}
\begin{enumerate}
\item
Assume that $\rank(\seq^0|\seq^1)=2$, $E_1$ satisfies neither \eqref{first E} nor \eqref{second E}, $c_0 \neq 0$ and $\ps{b}{\seq^0}{}=0$.
Then, condition \eqref{cond Eone r to five} holds if, and only if, there exists $j_0 \in \ens{1,\ldots,\newn}$ such that
$$
b_{j_0}=0, \quad \rank \Delta_{j_0}=1,
$$
where $\Delta_{j_0} \in \R^{2\newn \times 2}$ is given by
$$
\Delta_{j_0}
=\begin{pmatrix}
(D-d_{j_0}) \seq^0 & A^\tr e_{j_0} \\
A^\tr \seq^0 & v_{j_0} \seq^0 -c_0 e_{j_0}
\end{pmatrix},
$$
where $d_{j_0}$ is the $j_0$-th diagonal entry of $D$ and $e_{j_0}$ is the $j_0$-th canonical vector of $\R^{\newn}$.

\item\label{prop caract rank two seq i2}
Assume that $n=4$, $i=2$,
$$
b=\begin{pmatrix}
1 \\
0 \\
-\rho
\end{pmatrix},
\quad \rho \neq 0,
$$
$c_0 \neq 0$, $\alpha_{24} \neq 0$, $\ps{b}{\seq^0}{}=0$ and $\rank \Delta_2=1$.
Then, the solution to the kernel equations \eqref{kern equ vec} with $\delta=-1$ is given, for some $\sigma \in \R\setminus\ens{0}$, by
\begin{equation}\label{equ k21 k22 k23}
\begin{dcases}
k_{21}=k_{24} \rho,
\\
k_{22} =\frac{1}{-\sigma\alpha_{24}}\left(\sigma \mu_3 \px{k_{24}} +\sigma \nu_3 \pxi{k_{24}} -\alpha_{32}k_{24}\right), \\
k_{23} =\frac{1}{-\sigma\alpha_{24}}\left(\px{k_{24}}+\pxi{k_{24}} -\sigma\alpha_{23}k_{24}\right),
\end{dcases}
\end{equation}
where $k_{24} \in C^{\infty}(\R^2)$ is the solution to
\begin{equation}\label{wave equ k24}
\begin{dcases}
\mu_3 \pxx{k_{24}}(x,\xi)
+(\mu_3+\nu_3) \frac{\partial^2 k_{24}}{\partial x \partial \xi}(x,\xi)
+\nu_3 \pxixi{k_{24}}(x,\xi)
=
-2 c_0 k_{24}(x,\xi),
\\
\mu_3 \px{k_{24}}(x,-x)
+\nu_3 \pxi{k_{24}}(x,-x)
=\frac{\alpha_{32}}{\sigma} k_{24}(x,-x),
\\
k_{24}(x,x)=\alpha_{24}.
\end{dcases}
\end{equation}

\end{enumerate}

\end{proposition}

To see that \eqref{wave equ k24} indeed has a solution, we can introduce $h=\mu_3 \px{k_{24}}+\nu_3 \pxi{k_{24}}$ and observe that $(k_{24},h)$ satisfies a first-order hyperbolic system similar to the kernel equations \eqref{kern equ compo} and whose well-posedness can be established as in the proof of Theorem \ref{thm existence and analyticity}.

\begin{proof}[Proof of Proposition \ref{prop caract rank two seq}]
\begin{enumerate}
\item
Under the assumptions of the proposition and from the proof of Proposition \ref{thm caract E2}, we have $\seq^r \in \Eone$ for every $r \in \ens{2,\ldots, 5}$ if, and only if, there exist $\alpha_2, \beta_2 \in \R$ such that
$$
\begin{dcases}
\seq^2=\alpha_2 \seq^0 + \beta_2 \seq^1, \\
\seq^3=\alpha_3 \seq^0 + \beta_3 \seq^1,
\end{dcases}
$$
where $\alpha_3,\beta_3$ are given by \eqref{cond a3b3}.
We can check that this is equivalent to the existence of $\rho,\theta \in \R$ such that
\begin{equation}\label{cond with newc1}
\begin{dcases}
(A^\tr-\rho) \seq^1=-c_0 (D-\theta) \seq^0
+\newc_1(\seq^1-\rho\seq^0), \\
(D-\theta)(\seq^1-\rho \seq^0)=0,
\end{dcases}
\end{equation}
where we recall that $\newc_1=\frac{\ps{v}{\seq^1}{}}{c_0}$.
Since $D$ is diagonal with distinct entries and $\seq^0, \seq^1$ are linearly independent, the second condition in \eqref{cond with newc1} is equivalent to the existence of some $j_0 \in \ens{1,\ldots,\newn}$ and $r \in \R \setminus\ens{0}$ such that
\begin{equation}\label{cond from D}
\theta=d_{j_0}, \quad \seq^1=\rho \seq^0 +r e_{j_0}.
\end{equation}
Plugging the second identity in the first condition in \eqref{cond with newc1} and recalling that $\seq^1= A^\tr \seq^0$, we see that this condition simply becomes
$$
r A^\tr e_{j_0}=-c_0 (D-d_{j_0}) \seq^0
+r\newc_1 e_{j_0}.
$$

Comparing the $j_0$-th components of both quantities, using that the diagonal of $A$ is zero and $r \neq 0$, we see that $\newc_1=0$.
Recalling \eqref{cond from D}, the condition $\newc_1=0$ is equivalent to $\rho=-\newr v_{j_0}$ with $\newr=\frac{r}{c_0}$.
In summary, there exist $\rho,\theta \in \R$ such that \eqref{cond with newc1} holds if, and only if, there exist $j_0 \in \ens{1,\ldots,\newn}$ such that
\begin{equation}\label{equiv cond Delta}
\exists \newr \in \R \setminus\ens{0}, \quad \begin{dcases}
(D-d_{j_0}) \seq^0 + \newr A^\tr e_{j_0}=0, \\
\seq^1 + \newr (v_{j_0} \seq^0 -c_0 e_{j_0})=0.
\end{dcases}
\end{equation}
We can check that this condition is equivalent to $\rank \Delta_{j_0}=1$.
Finally, it is clear that $E_1 \subset \ker b^\tr$, i.e. $\ps{b}{\seq^1}{}=0$, if, and only if, $b_{j_0}=0$.

\item
Since $i=2$, we have (for the notations, see Section \ref{sect notaa})
$$
A^\tr=
\begin{pmatrix}
0 & \alpha_{31} & \alpha_{41} \\
\alpha_{13} & 0 & \alpha_{43} \\
\alpha_{14} & \alpha_{34} & 0
\end{pmatrix}
, \quad
D=2\diag(\mu_1,\mu_3,\mu_4) -\Id_3
, \quad
\seq^0=
\begin{pmatrix}
\alpha_{21} \\
\alpha_{23} \\
\alpha_{24}
\end{pmatrix},
\quad
v=
-\frac{1}{2}
\begin{pmatrix}
\alpha_{12} \\
\alpha_{32} \\
\alpha_{42}
\end{pmatrix}.
$$
Clearly, $\ps{b}{\seq^0}{}=0$ is equivalent to
\begin{equation}\label{cond a21}
\alpha_{21} = \alpha_{24}\rho.
\end{equation}
On the other hand, using the characterization \eqref{equiv cond Delta}, we see that $\rank \Delta_2=1$ if, and only if,
\begin{align}
\alpha_{31} &= \sigma(\mu_1-\mu_3)\rho \alpha_{24},
\label{cond a31} \\
\alpha_{34} &=  \sigma(\mu_4-\mu_3)\alpha_{24},
\label{cond a34} \\
\alpha_{41} &= -\frac{\rho}{\sigma}\left(\alpha_{32} +\sigma^2 (\mu_1-\mu_3)\alpha_{23}\right),
\label{cond a41} \\
\alpha_{42}+\alpha_{12}\rho &=\sigma (\alpha_{43}+\alpha_{13}\rho),
\label{cond a43 a13} \\
\alpha_{14} &=-\frac{1}{\rho \sigma}\left(\alpha_{32} + \sigma^2(\mu_4-\mu_3)\alpha_{23}\right),
\label{cond a14}
\end{align}
where $\sigma=-\frac{2}{\newr}$.
Using these conditions we easily check that $((\newk)^\tr,k_2) =(k_{21},k_{23},k_{24},k_{22})$ defined by \eqref{equ k21 k22 k23}-\eqref{wave equ k24} satisfies the kernel equations \eqref{kern equ vec} with $\delta=-1$.

\end{enumerate}

\end{proof}

To conclude this section we will present a method which shows how conditions \eqref{cond a21}-\eqref{cond a14} can also be found from an analytic point of view.

\begin{proof}[Another proof of Proposition \ref{prop caract rank two seq}, item \ref{prop caract rank two seq i2}]
\begin{enumerate}
\item
For every $j$, let us denote by $\opP_j$ the first-order linear partial differential operator
$$\opP_j=\mu_j\px{}+\pxi{}\nu_j.$$

Since we want the orthogonality condition $\ps{b}{\newk(\cdot,0)}{}=0$, we look for a solution satisfying
$$k_{21}=k_{24}\rho,$$
(recall also Remark \ref{rem equiv char}).
In particular, we assume \eqref{cond a21}.
Then, the problem is to find a solution to
\begin{equation}\label{kern equ Pnot}
\begin{dcases}
\begin{aligned}
k_{22}\alpha_{21}+k_{23}\alpha_{31} &=(-\rho\opP_1-\alpha_{41}) k_{24},
\\
\opP_2 k_{22} +k_{23}\alpha_{32}+k_{24}(\alpha_{42}+\alpha_{12}\rho) &=0,
\\
\opP_3 k_{23} +k_{22}\alpha_{23}+k_{24}(\alpha_{43}+\alpha_{13}\rho) &=0,
\\
k_{22}\alpha_{24}+k_{23}\alpha_{34} &= (-\opP_4-\alpha_{14}\rho) k_{24}.
\end{aligned}
\end{dcases}
\end{equation}

Let us denote by
$$
\omega=\det \begin{pmatrix}
\alpha_{21} & \alpha_{31} \\
\alpha_{24} & \alpha_{34}
\end{pmatrix}.
$$
Assume that $\omega \neq 0$ (this will follow a posteriori from \eqref{cond a31}, \eqref{cond a34}, using also that $\sigma,\rho,\alpha_{24} \neq 0$).
Then, the first and fourth equations in \eqref{kern equ Pnot} give
\begin{equation}\label{def k22 k23}
\begin{dcases}
k_{22} =\frac{1}{\omega}\left(-\alpha_{34}\rho\opP_1 +\alpha_{31}\opP_4
-\alpha_{34}\alpha_{41} +\alpha_{31} \alpha_{14}\rho\right) k_{24}, \\
k_{23} =\frac{1}{\omega}\left(\alpha_{24}\rho \opP_1 -\alpha_{21}\opP_4
+\alpha_{24}\alpha_{41} -\alpha_{21} \alpha_{14} \rho\right) k_{24}.
\end{dcases}
\end{equation}

Plugging these relations in the second and third equations in \eqref{kern equ Pnot} give the following two second-order partial differential equations for $k_{24}$:
$$
\opQ k_{24}=0, \quad \tilde{\opQ}  k_{24}=0,
$$
where $\opQ=\opQ^{(2)} +\opQ^{(1)} +\opQ^{(0)}$, with
\begin{align*}
\opQ^{(2)} &=\opP_2\left(-\alpha_{34}\rho\opP_1 +\alpha_{31}\opP_4\right),
\\
\opQ^{(1)} &=(-\alpha_{34}\alpha_{41} +\alpha_{31} \alpha_{14} \rho)\opP_2
+\left(\alpha_{24}\rho \opP_1 -\alpha_{21} \opP_4\right) \alpha_{32},
\\
\opQ^{(0)} &=(\alpha_{24}\alpha_{41} -\alpha_{21} \alpha_{14} \rho) \alpha_{32}+\omega (\alpha_{42}+\alpha_{12}\rho),
\end{align*}
and $\tilde{\opQ}=\tilde{\opQ}^{(2)} +\tilde{\opQ}^{(1)} +\tilde{\opQ}^{(0)}$, with
\begin{align*}
\tilde{\opQ}^{(2)} &=\opP_3 \left(\alpha_{24}\rho\opP_1 -\alpha_{21}\opP_4\right),
\\
\tilde{\opQ}^{(1)} &=(\alpha_{24}\alpha_{41} -\alpha_{21} \alpha_{14} \rho) \opP_3
+\left(-\alpha_{34}\rho\opP_1 +\alpha_{31}\opP_4\right)\alpha_{23},
\\
\tilde{\opQ}^{(0)} &=(-\alpha_{34}\alpha_{41} +\alpha_{31} \alpha_{14} \rho) \alpha_{23} +\omega(\alpha_{43}+\alpha_{13}\rho).
\end{align*}

\item
We are going to find conditions to guarantee that these two equations are compatible.
To this end, it is for instance sufficient to have
\begin{equation}\label{cond diff op r}
\opQ^{(r)}= \sigma \tilde{\opQ}^{(r)}, \quad r=0,1,2,
\end{equation}
for some $\sigma \in \R$.
We first look at the operators of highest order.
Using the identities
\begin{equation}\label{useful ids}
\alpha_{21}=\alpha_{24}\rho, \quad \mu_j-\nu_j=1 \quad (j \neq 2), \quad \mu_2=\nu_2=1,
\end{equation}
we have
\begin{align}
\alpha_{24}\rho\opP_1 -\alpha_{21}\opP_4
&=\alpha_{24}\rho \left(\opP_1-\opP_4\right)
\nonumber
\\
&=\alpha_{24}\rho (\mu_1-\mu_4) \opP_2.
\label{simp Q2}
\end{align}
It follows that
\begin{equation}\label{simp expr Q2}
\tilde{\opQ}^{(2)}=\alpha_{24}\rho (\mu_1-\mu_4) \opP_3 \opP_2.
\end{equation}

Consequently, we see that \eqref{cond diff op r} holds for $r=2$ if we have
\begin{equation}\label{equ operators 2}
-\alpha_{34}\rho\opP_1 +\alpha_{31}\opP_4
=\sigma \alpha_{24}\rho (\mu_1-\mu_4) \opP_3.
\end{equation}
This identity holds if $(\alpha_{34},\alpha_{31})$ satisfies
$$
\begin{pmatrix}
-\rho \mu_1 & \mu_4 \\
-\rho \nu_1 & \nu_4
\end{pmatrix}
\begin{pmatrix}
\alpha_{34} \\
\alpha_{31}
\end{pmatrix}
=\sigma \alpha_{24}\rho(\mu_1-\mu_4)
\begin{pmatrix}
\mu_3 \\
\nu_3
\end{pmatrix},
$$
which is equivalent to \eqref{cond a34}-\eqref{cond a31} (using $\rho \neq 0$ and \eqref{useful ids}).

\item
Let us now compute the first-order differential operators.
We have
\begin{align*}
\opQ^{(1)} &=\left(
-\alpha_{34}\alpha_{41} +\alpha_{31} \alpha_{14} \rho
+\alpha_{24}\rho (\mu_1-\mu_4) \alpha_{32}
\right)\opP_2 \quad \text{ (by \eqref{simp Q2}), }
\\
\tilde{\opQ}^{(1)} &=\left(
\alpha_{24}\alpha_{41} -\alpha_{21} \alpha_{14} \rho
+\sigma \alpha_{24}\rho(\mu_1-\mu_4) \alpha_{23}
\right)\opP_3 \quad \text{ (by \eqref{equ operators 2}). }
\end{align*}
As a result, we have \eqref{cond diff op r} for $r=1$, if $\opQ^{(1)}=\tilde{\opQ}^{(1)}=0$,
that is, if
\begin{equation}\label{rela order zero}
\begin{dcases}
-\alpha_{34}\alpha_{41} +\alpha_{31} \alpha_{14} \rho
=-\alpha_{24}\rho (\mu_1-\mu_4) \alpha_{32},
\\
\alpha_{24}\alpha_{41} -\alpha_{21} \alpha_{14} \rho
=-\sigma \alpha_{24}\rho(\mu_1-\mu_4) \alpha_{23}.
\end{dcases}
\end{equation}
This holds if $(\alpha_{41},\alpha_{14})$ satisfies
$$
\begin{pmatrix}
-\alpha_{34} & \alpha_{31}\rho \\
\alpha_{24} & -\alpha_{21}\rho
\end{pmatrix}
\begin{pmatrix}
\alpha_{41} \\
\alpha_{14}
\end{pmatrix}
=-\alpha_{24}\rho(\mu_1-\mu_4)
\begin{pmatrix}
\alpha_{32} \\
\sigma \alpha_{23}
\end{pmatrix},
$$
which is equivalent to \eqref{cond a41} and \eqref{cond a14} (using \eqref{cond a21}, \eqref{cond a31}, \eqref{cond a34} and $\alpha_{24}, \rho \neq 0$).

\item
Let us now compute the zero order terms.
Using \eqref{rela order zero}, we immediately see that
\begin{align*}
\tilde{\opQ}^{(0)} &=
-\alpha_{24}\rho(\mu_1-\mu_4) \alpha_{32} \alpha_{23}
+\omega (\alpha_{43}+\alpha_{13}\rho),
\\
\opQ^{(0)} &=
-\sigma\alpha_{24}\rho(\mu_1-\mu_4) \alpha_{23} \alpha_{32}
+\omega (\alpha_{42}+\alpha_{12}\rho).
\end{align*}
As a result, we see that \eqref{cond diff op r} holds for $r=0$ if we have condition \eqref{cond a43 a13}.
Moreover, using \eqref{cond a21}, \eqref{cond a31} and \eqref{cond a34}, we have
\begin{equation}\label{comp delta}
\omega=-\sigma \alpha_{24}^2 \rho (\mu_1-\mu_4),
\end{equation}
so that, using again \eqref{cond a21} and the definition of $c_0$, we obtain
$$\opQ^{(0)}=\sigma\alpha_{24}\rho(\mu_1-\mu_4) (2c_0).$$

It follows that $k_{24}$ indeed satisfies the first equation in \eqref{wave equ k24} (recall that $\opQ^{(2)}=\sigma\tilde{\opQ}^{(2)}$ with \eqref{simp expr Q2} and $\opQ^{(1)}=0$).

\item
Using \eqref{equ operators 2}, \eqref{rela order zero}, \eqref{simp Q2} and \eqref{comp delta}, we can simplify the expressions in \eqref{def k22 k23} to obtain
$$
\begin{dcases}
k_{22} =\frac{1}{-\sigma\alpha_{24}}\left(\sigma \opP_3 -\alpha_{32}\right) k_{24}, \\
k_{23} =\frac{1}{-\sigma\alpha_{24}}\left(\opP_2 -\sigma\alpha_{23}\right) k_{24}.
\end{dcases}
$$

In addition, it follows from these formula that the remaining conditions are satisfied.
Indeed, the condition $k_{22}(x,-x)=0$ is exactly the condition that we require for $k_{24}$ at $(x,-x)$ in \eqref{wave equ k24} and the condition $k_{23}(x,x)=\alpha_{23}$ follows from the above expression since $k_{24}(x,x)=\alpha_{24}$ and $(\opP_2 k_{24})(x,x)=\frac{d}{dx} k_{24}(x,x)=0$.

\end{enumerate}
\end{proof}

\begin{remark}
In \cite[Section 3.3]{VK14}, the authors showed that we can solve a kernel system of two equations of the form
$$
\begin{dcases}
\px{k_{21}} -\pxi{k_{21}} +k_{22}\alpha_{21}=0, \\
\px{k_{22}} +\pxi{k_{22}} +k_{21}\alpha_{12}=0, \\
k_{21}(x,x)=\alpha_{21}, \quad k_{22}(x,0)=0,
\end{dcases}
$$
with $\alpha_{21} \neq 0$ by first expressing $k_{22}$ from the first equation and then showing that the resulting second order equation for $k_{21}$ indeed has a solution.
The method we introduced in the second proof of Proposition \ref{prop caract rank two seq}, item \ref{prop caract rank two seq i2}, can be seen as an extension of the method of \cite{VK14} where, instead of dividing by a scalar (namely, $\alpha_{21}$), we invert a matrix.
\end{remark}

\section*{Acknowledgements}

This project was supported by National Natural Science Foundation of China (Nos. 12122110 and 12071258) and National Science Centre, Poland UMO-2020/39/D/ST1/01136.
For the purpose of Open Access, the authors have applied a CC-BY public copyright licence to any Author Accepted Manuscript (AAM) version arising from this submission.

\appendix
\section{Controllability of the equivalent system}\label{app thm G}

In this appendix, we give a simple and direct proof of Corollary \ref{thm G}.
We recall that it can be deduced from Theorem \ref{thm HO21} but this result is based on the Titchmarsh convolution theorem (see \cite{HO21-COCV}) and we show here how to directly prove the corollary without resorting to this difficult result.

\begin{proof}[Proof of Corollary \ref{thm G}]
It is enough to show that, if $(q,f) \in \mathcal{S}_k \setminus \mathcal{S}_{k+1}$ for some $k \in \ens{2,\ldots,n+1}$, then system \eqref{syst G} (with $m=1$) is null controllable in time $T$ if, and only if, $T \geq \tau_k$.
\begin{enumerate}
\item
We first observe that system \eqref{syst G} is equivalent to the same system with $f_1=0$.
This follows from the invertible spatial transformation
$$
\hat{y}_1(t,x)=\tilde{y}_1(t,x)-\int_0^x h(x-\xi)\tilde{y}_1(t,\xi) \, d\xi,
$$
where the kernel $h$ is the solution to
$$
h(x)\lambda_1+\int_0^x h(x-\xi) f_1(\xi) \, d\xi=f_1(x), \quad 0<x<1.
$$
Therefore, for the rest of the proof, we assume that $f_1=0$.

\item
Assume now that $(q,f) \in \mathcal{S}_k \setminus \mathcal{S}_{k+1}$ for some $k \in \ens{2,\ldots,n}$ (the result for $k=n+1$ is trivial).
It will be convenient to use the notation $\qshift_i=q_{i-1}$ for $2 \leq i \leq n$.
Let us write system \eqref{syst G} (with $f_1=0$) component-wise:
$$
\begin{dcases}
\pt{\tilde{y}_1}(t,x)+\lambda_1 \px{\tilde{y}_1}(t,x)=0, \\
\tilde{y}_1(t,1)=\tilde{u}(t),  \\
\tilde{y}_1(0,x)=\tilde{y}_1^0(x),
\end{dcases}
\quad
\begin{dcases}
\pt{\tilde{y}_i}(t,x)+\lambda_i \px{\tilde{y}_i}(t,x)=f_i(x)v(t), \\
\tilde{y}_i(t,0)=\qshift_i v(t), \\
\tilde{y}_i(0,x)=\tilde{y}_i^0(x),
\end{dcases}
$$
for $i \in \ens{2,\ldots,n}$, and where we introduced $v(t)=\tilde{y}_1(t,0)$.
It is clear that this system is null controllable in any time $T \geq \tau_k=\max\ens{T_1+T_k,T_2}$ since in this case taking $\tilde{u}=0$ in $(T-(T_1+T_k),T)$ does the job.
It is the necessary part that requires more work.

\item
First of all, we recall that the condition $T \geq \max\ens{T_1,T_2}$ is always necessary (see e.g. the proof of \cite[Lemma 3.3]{HO21-COCV}).
Under this condition and by mimicking the second step in the proof of \cite[Theorem 3.1]{HO21-COCV}, we see that the null controllability condition $\tilde{y}_k(T,x)=0$ is equivalent to

\begin{equation}\label{cns yd}
\qshift_k\alpha(\tau)
+\int_0^{\tau} \beta(\tau-\sigma) \alpha(\sigma) \, d\sigma=0,
\quad 0<\tau<T_k,
\end{equation}
where $\alpha(\theta)=v(-\theta+T)$ and $\beta(\theta)=f_k(\lambda_k \theta)$ for $0<\theta<T_k$.

\item
We now have two possibilities for \eqref{cns yd}.

\begin{enumerate}
\item
Case $\qshift_k \neq 0$.
Then, by uniqueness of the solution to the Volterra equation of the second kind \eqref{cns yd}, we obtain $\alpha=0$ in $(0,T_k)$.
This means that $v=0$ in $(T-T_k, T)$.
Since this is true for any $\tilde{y}_0^1$, it is possible only if $T_1 \leq T-T_k$, which is the desired condition.

\item
Case $\qshift_k=0$.
Since $(q,f) \in \mathcal{S}_k \setminus \mathcal{S}_{k+1}$, we necessarily have $f_k \neq 0$.
Since $f_k$ is analytic in a neighborhood of $[0,1)$, this implies in particular that there exists $N \geq 1$ such that
$$
f_k^{(N-1)}(0) \neq 0, \quad
f_k^{(\ell)}(0)=0, \quad \forall \ell<N-1.
$$

Then, taking $N$ times the derivative with respect to $\tau$ in \eqref{cns yd} (with $\qshift_k=0$), we obtain the new Volterra equation
$$
c\alpha(\tau)
+\int_0^{\tau} \beta^{(N)}(\tau-\sigma) \alpha(\sigma) \, d\sigma=0,
\quad 0<\tau<T_k,
$$
where $c=\beta^{(N-1)}(0)=f_k^{(N-1)}(0) \lambda_k^{N-1}$.
Therefore, $c \neq 0$ and the situation is now identical to the previous case.

\end{enumerate}

\end{enumerate}

\end{proof}

\section{Solution to the kernel equations}\label{app thm K}

In this appendix, we present a new approach to solve the kernel equations that encompasses in particular the proof of Theorem \ref{thm existence and analyticity}.
We recall that, when considering the kernel equations in the triangle $\Tau=\ens{(x,\xi) \in \R^2 \st 0<\xi<x<1}$, the approach used in all current results in the literature (\cite{CVKB13,DMVK13,HDM15,HDMVK16,HVDMK19,CN19}, etc.) consists in adding ``artificial boundary conditions'' to close the system of kernel equations.
In our approach, we will not consider the condition at $(x,x)$ as a boundary condition but rather as an initial condition.
We will simply let propagate this condition along the characteristics and find the corresponding so-called domain of determinacy, much in the spirit of the reference books \cite{LY85,Bre00}.
Then, another idea of our method is also to solve the equation for $j=i$ and plug it into the other equations of the system to obtain a new system with initial conditions at $(x,x)$ only (as in the proof of Theorem \ref{thm derivatives}).
Moreover, this gives a natural bound in $\abs{x-\xi}$ for the estimates needed to prove the contraction of the mapping defining the integral equations corresponding to the new system (rather than $\abs{x-(1-\epsilon)\xi}$ as in \cite{HDMVK16,HVDMK19}).

All along this appendix, $i \in \ens{1,\ldots,n}$ is fixed and we continue using the notation $k=(k_{ij})_{1 \leq j \leq n}$ to denote the transpose of the $i$-th row of $K$.
We also emphasize that $m \geq 1$ is arbitrary.

First of all, it will be more convenient to work with the kernel equations normalized by $\lambda_i$:
\begin{equation}\label{kern equ compo b}
\begin{dcases}
\px{k_j}(x,\xi)
+\lambdan_j\pxi{k_j}(x,\xi) 
+\sum_{r=1}^n k_r(x,\xi) \mn_{rj}=0,
\\
k_j(x,x)=f_j \quad (j \neq i),
\quad k_i(x,\delta x)=0,
\end{dcases}
\end{equation}
where
$$
\lambdan_j=\frac{\lambda_j}{\lambda_i}
, \quad
\mn_{rj}=\frac{m_{rj}}{\lambda_i}
, \quad
f_j=
\frac{m_{ij}}{\lambda_i-\lambda_j}.
$$

From now on, we will assume for instance that $i \geq m+1$, so that $\lambda_i>0$ and thus, from \eqref{hyp speeds},
\begin{equation}\label{new speed order}
\lambdan_1<\cdots<\lambdan_{i-1}<1<\lambdan_{i+1}<\cdots<\lambdan_n.
\end{equation}
For every $(x,\xi) \in \R \times \R$, we denote by $s \mapsto \zeta_j(s;x,\xi)$ the solution to
$$
\begin{dcases}
\dds \zeta_j(s;x,\xi)=\lambdan_j, \quad \forall s \in \R, \\
\zeta_j(x;x,\xi)=\xi.
\end{dcases}
$$

Let us now consider the more general condition
$$k_j(x,x)=f_j(x) \quad (j \neq i),$$
where $f_j$ is a function defined on an interval of the form $[a,b]$ with $a<0<b$.
Even if $f_j$ is constant in \eqref{kern equ compo b}, we will need to consider space-dependent data to deduce the existence of smooth solutions by an inductive argument.
We will describe the largest domain $D \subset \R^2$ where the system can then be solved along the characteristics.
We first take care of the characteristics for $j \neq i$.
Recalling the ordering \eqref{new speed order}, we introduce
$$
D^{\mathfrak{c}}=\ens{(x,\xi) \in \R^2 \st
\begin{array}{l}
\zeta_{i-1}(x;a,a)<\xi<\zeta_{i-1}(x;b,b)
\\
\zeta_{i+1}(x;b,b)<\xi<\zeta_{i+1}(x;a,a)
\end{array}
},
$$
(see Figure \ref{figure domain Dc}).
Above, we use the usual conventions for $i=1$ and $i=n$.
We now take care of the characteristic for $j=i$.
We can check that the line $\ens{(x,\delta x) \st a<x<b}$ intersects the boundary of $D^{\mathfrak{c}}$ at exactly two points $(c,\delta c)$ and $(d,\delta d)$, with $c,d \in (a,b)$ and $c<0<d$ if $\delta<1$ or $d<0<c$ if $\delta>1$.
Let then
$$
D=\ens{(x,\xi) \in D^{\mathfrak{c}} \st
\zeta_i(x;d,\delta d)< \xi < \zeta_i(x;c,\delta c)
},
$$
(see Figure \ref{figure domain D} with $\delta=-1$) and define $I=(c,d)$.
Here and in what follows, it will be convenient to use the notation $(\alpha,\beta)$ to denote the interval $(\min\ens{\alpha,\beta},\max\ens{\alpha,\beta})$, whatever $\alpha,\beta \in \R$ are (we use a similar notation for $[\alpha,\beta]$).

\def\figlambdaone{-2}
\def\figlambdatwo{-1/5}
\def\figlambdathree{1}
\def\figlambdafour{1.5}
\def\figa{-0.75}
\def\figb{1.15}
\def\xtop{0.5911765}
\def\ytop{1.2617647}
\def\xbot{-0.1911765}
\def\ybot{-0.8617647}
\def\Ymax{2.1}
\def\Ymin{-1.7}

\def\figdelta{-1}
\def\figc{-0.15}
\def\xtopbis{0.9}
\def\ytopbis{1.2}

\def\figd{0.23}
\def\xbotbis{-0.3666667}
\def\ybotbis{-0.8266667}

\begin{figure}[h!]
\begin{minipage}[c]{0.5\textwidth}
\centering

\begin{tikzpicture}[scale=1.5]
\draw[->] (-1,0)--(1.4,0) node[right]{$x$};
\draw[->] (0,\Ymin-0.25)--(0,\Ymax+0.25) node[above]{$\xi$};

\node at (\figa,0) (a) {};
\node [below left=0.05cm and 0.05cm of a] {$a$};
\node at (\figb,0) (b) {};
\node [below right=0.05cm and 0.05cm of b] {$b$};
\draw[dashed] (\figa,\Ymin-0.25)--(\figa,\Ymax+0.25);
\draw[dashed] (\figb,\Ymin-0.25)--(\figb,\Ymax+0.25);

\fill[gray!60,opacity=0.4] (\figa,\figa) -- (\xtop,\ytop) -- (\figb,\figb) -- (\xbot,\ybot) -- (\figa,\figa);

\draw[thick, domain=\figa:\figb] plot (\x, {\x});
\draw[ultra thick, domain=\figa:\figb] plot (\x, {\figlambdatwo*(\x-\figa)+\figa});
\draw[ultra thick, domain=\figa:\figb] plot (\x, {\figlambdatwo*(\x-\figb)+\figb});
\draw[ultra thick, domain=\figa:\figb] plot (\x, {\figlambdafour*(\x-\figb)+\figb});
\draw[ultra thick, domain=\figa:\figb] plot (\x, {\figlambdafour*(\x-\figa)+\figa});

\end{tikzpicture}

\captionof{figure}{Domain $D^\mathfrak{c}$}\label{figure domain Dc}
\end{minipage}
\begin{minipage}[c]{0.5\textwidth}
\centering

\begin{tikzpicture}[scale=1.5]
\draw[->] (-1,0)--(1.4,0) node[right]{$x$};
\draw[->] (0,\Ymin-0.25)--(0,\Ymax+0.25) node[above]{$\xi$};

\node at (\figa,0) (a) {};
\node [below left=0.05cm and 0.05cm of a] {$a$};
\node at (\figb,0) (b) {};
\node [below right=0.05cm and 0.05cm of b] {$b$};
\draw[dashed] (\figa,\Ymin-0.25)--(\figa,\Ymax+0.25);
\draw[dashed] (\figb,\Ymin-0.25)--(\figb,\Ymax+0.25);

\fill[gray!60,opacity=0.4] (\figa,\figa) -- (\xtop,\ytop) -- (\figb,\figb) -- (\xbot,\ybot) -- (\figa,\figa);

\draw[domain=\figa:\figb] plot (\x, {\x});
\draw[domain=\figa:\figb] plot (\x, {\figlambdatwo*(\x-\figa)+\figa});
\draw[domain=\figa:\figb] plot (\x, {\figlambdatwo*(\x-\figb)+\figb});
\draw[domain=\figa:\figb] plot (\x, {\figlambdafour*(\x-\figb)+\figb});
\draw[domain=\figa:\figb] plot (\x, {\figlambdafour*(\x-\figa)+\figa});

\fill[gray!60,opacity=0.8] (\figa,\figa) -- (\figc,\figdelta*\figc) -- (\xtopbis,\ytopbis) -- (\figb,\figb) -- (\figd,\figdelta*\figd) -- (\xbotbis,\ybotbis) -- (\figa,\figa);

\draw[thick, domain=\figa:\figb] plot (\x, {-\x});
\draw[ultra thick, domain=\figa:\figb] plot (\x, {\figlambdathree*(\x-\figc)-\figc});
\draw[ultra thick, domain=\figa:\figb] plot (\x, {\figlambdathree*(\x-\figd)-\figd});

\draw[dashed] (\figc,\figdelta*\figc) -- (\figc,0) node {$\times$} node[below] {$c$};

\draw[dashed] (\figd,\figdelta*\figd) -- (\figd,0) node {$\times$} node[above] {$d$};

\end{tikzpicture}

\captionof{figure}{Domain $D$ (in dark gray)}\label{figure domain D}
\end{minipage}
\end{figure}

We will prove the following result.

\begin{theorem}\label{thm exist kern}
Let $a<0<b$ and $s \in \N$ be fixed.
For any $(f_j)_{j \neq i} \in C^s([a,b])^{n-1}$ and $f_i \in C^s(\clos{I})$, there exists a unique solution $k=(k_j)_{1 \leq j \leq n} \in C^s(\clos{D})^n$ to
\begin{equation}\label{kern equ with BC}
\begin{dcases}
\px{k_j}(x,\xi)
+\lambdan_j\pxi{k_j}(x,\xi) 
+\sum_{r=1}^n k_r(x,\xi) \mn_{rj}=0, \quad (x,\xi) \in D,
\\
k_j(x,x)=f_j(x), \quad x \in (a,b) \quad (j \neq i), \\
k_i(x,\delta x)=f_i(x), \quad x \in I,
\end{dcases}
\end{equation}
Moreover, we have the estimate
\begin{equation}\label{thm exist kern estimate}
\norm{k}_{C^s(\clos{D})^n} \leq C  \max\ens{
\max_{j \neq i}\norm{f_j}_{C^s([a,b])},
\quad
\norm{f_i}_{C^s(\clos{I})}
},
\end{equation}
for some $C>0$ that does not depend on any $f_j$.

\end{theorem}

For $s=0$, by solution we mean ``solution along the characteristics'', see below.

The first part of Theorem \ref{thm existence and analyticity} follows from the previous result and the following simple observation:
\begin{equation}\label{basic prop D}
\forall V \subset \R^2, \exists a<0<b, \quad V \subset D.
\end{equation}

On the other hand, using that the coefficients of the system are constant and arguing as in the proof of \cite[Lemma 6.2]{CN19}, we can show that \eqref{thm exist kern estimate} holds with $C=R^s$ for some $R>0$ that does not depend on $s$.
This establishes the estimate in Theorem \ref{thm existence and analyticity}.

Let us now prove Theorem \ref{thm exist kern}.
We start with a description of the key properties satisfied by the point where the $j$-th characteristic intersects the corresponding data line.

\begin{lemma}\label{prop char remain in D}
For every $j \in \ens{1,\ldots,n}$, there exists $\sigma_j \in C^{\infty}(\clos{D})$ such that, for every $(x,\xi) \in \clos{D}$, we have:
\begin{itemize}
\item
$\zeta_j(\sigma_j(x,\xi);x,\xi)=\sigma_j(x,\xi)$ with $\sigma_j(x,\xi) \in [a,b]$ for $j \neq i$ and $\zeta_i(\sigma_i(x,\xi);x,\xi)=\delta \sigma_i(x,\xi)$ with $\sigma_i(x,\xi) \in \clos{I}$.

\item
$(s,\zeta_j(s;x,\xi)) \in \clos{D}$ for every $s \in [\sigma_j(x,\xi), x]$.

\item
For every $j \neq i$, we have
\begin{equation}\label{estim sigmaj-x}
\abs{\sigma_j(x,\xi)-x} \leq C \abs{x-\xi},
\end{equation}
for some $C>0$ that does not depend on $j,x,\xi$.
\end{itemize}

\end{lemma}

We point out that $\zeta_j$ and $\sigma_j$ are explicit.
In particular, this is how we prove estimate \eqref{estim sigmaj-x}.

Now, instead of writing \eqref{kern equ with BC} along all the characteristics (as it is usually done), we first replace $k_i$ by formally solving the corresponding equation (recall that $\mn_{ii}=0$):
\begin{equation}\label{equation ki}
k_i(x,\xi)=f_i(\sigma_i(x,\xi))
-\int_{\sigma_i(x,\xi)}^x \sum_{r \neq i} k_r(\eta,\zeta_i(\eta;x,\xi)) \mn_{ri} \, d\eta.
\end{equation}
Let us introduce the following notations to exclude the $i$-th components: $\knoi=(k_j)_{j \neq i}$, $\fnoi=(f_j)_{j \neq i}$, $\sigmanoi=(\sigma_j)_{j \neq i}$, $\zetanoi=(\zeta_j)_{j \neq i}$, $\Mnoi=(\mn_{rj})_{r,j \neq i}$, $\psi=(\mn_{ij})_{j \neq i}$ and $w=(\mn_{ji})_{j \neq i}$.
Then, plugging the previous expression of $k_i$ in \eqref{kern equ with BC} and integrating along the characteristics, we can transform this system into the following system of integral equations for $\knoi$:
\begin{align}
\knoi_{\ell}(x,\xi)
=& \fnoi_\ell(\sigmanoi_\ell(x,\xi))
-\int_{\sigmanoi_\ell(x,\xi)}^x \sum_{r=1}^{n-1} \knoi_r(s,\zetanoi_\ell(s;x,\xi))\mnoi_{r \ell} \, ds
\nonumber
\\
&
- \int_{\sigmanoi_\ell(x,\xi)}^x f_i(\sigmanoi_i(s,\zetanoi_\ell(s;x,\xi))) \psi_\ell \, ds 
\nonumber
\\
& + \int_{\sigmanoi_\ell(x,\xi)}^x \left(\int_{\sigmanoi_i(s,\zetanoi_\ell(s;x,\xi))}^s \sum_{r=1}^{n-1} \knoi_r(\eta,\zeta_i(\eta;s,\zetanoi_\ell(s;x,\xi))) w_r \, d\eta\right) \psi_\ell \, ds,
\label{int equ}
\end{align}
for every $\ell \in \ens{1,\ldots,n-1}$ and $(x,\xi) \in \clos{D}$.

All the quantities in \eqref{equation ki} and \eqref{int equ} are well defined thanks to Lemma \ref{prop char remain in D}.
It remains to prove the existence and uniqueness of a $C^s$ solution $\knoi$ to this system of integral equations.
We start with $s=0$.
As usual, we use the Banach fixed point theorem and suitable estimates.
A solution to this system is a fixed point of the map $F(\knoi)=u^0+\Phi \knoi$, where
$$
u^0_\ell(x,\xi)=
\fnoi_\ell(\sigmanoi_\ell(x,\xi))
- \int_{\sigmanoi_\ell(x,\xi)}^x f_i(\sigmanoi_i(s,\zetanoi_\ell(s;x,\xi)) \psi_\ell \, ds,
$$
and $\Phi$ is the linear map $\Phi=\Phi_1+\Phi_2$ with
$$
(\Phi_1 \knoi)_\ell(x,\xi)
=
-\int_{\sigmanoi_\ell(x,\xi)}^x \sum_{r=1}^{n-1} \knoi_r(s,\zetanoi_\ell(s;x,\xi))\mnoi_{r \ell} \, ds,
$$
and
$$
(\Phi_2 \knoi)_\ell(x,\xi)
=\int_{\sigmanoi_\ell(x,\xi)}^x \left(\int_{\sigmanoi_i(s,\zetanoi_\ell(s;x,\xi))}^s \sum_{r=1}^{n-1} \knoi_r(\eta,\zeta_i(\eta;s,\zetanoi_\ell(s;x,\xi))) w_r \, d\eta\right) \psi_\ell \, ds,
$$
for every $\ell \in \ens{1,\ldots,n-1}$ and $(x,\xi) \in \clos{D}$.

Let us now precisely set the functional framework.
Let $B=C^0(\clos{D})^{n-1}$ and consider the standard norm
$\norm{\knoi}_B=\max_{1 \leq \ell \leq n-1} \max_{(x,\xi) \in \clos{D}} \abs{\knoi_\ell(x,\xi)}$.
Clearly, $B$ is a Banach space and $F(B) \subset B$.
Let us now prove that $F^N$ is a contraction for $N \in \N^*$ large enough.
This is equivalent to show that $\Phi^N$ is a contraction.
To this end, it is sufficient to prove the following key estimate:

\begin{lemma}
There exists $C>0$ such that, for every $N \in \N^*$, we have
$$
\abs{(\Phi^N \knoi)_\ell(x,\xi)}
\leq
\frac{C^N \abs{x-\xi}^N}{N!}
\norm{\knoi}_B,
$$
for every $\knoi \in B$, $\ell \in \ens{1,\ldots,n-1}$ and $(x,\xi) \in \clos{D}$.

\end{lemma}

\begin{proof}
We prove the property by induction on $N$.
Let us first consider $N=1$.
We have
$$
\abs{(\Phi_1 \knoi)_\ell(x,\xi)} \leq C_1 \abs{x-\sigmanoi_\ell(x,\xi)} \norm{\knoi}_B,
$$
with $C_1=\max_{\ell} \sum_{r} \abs{\mnoi_{r\ell}} \geq 0$.
Similarly,
$$
\abs{(\Phi_2 \knoi)_j(x,\xi)} \leq C_2 \abs{x-\sigmanoi_\ell(x,\xi)} \norm{\knoi}_B,
$$
with $C_2=\max_{\ell, (x,\xi), s} \abs{s-\sigmanoi_i(s,\zetanoi_\ell(s;x,\xi))} \sum_{r} \abs{w_r} \abs{\psi_\ell} \geq 0$.
Finally, we have
\begin{equation}\label{estim x minus sigmax}
\abs{x-\sigmanoi_\ell(x,\xi)} \leq C_3 \abs{x-\xi},
\end{equation}
for some $C_3>0$ that does not depend on $\ell,x,\xi$ (see \eqref{estim sigmaj-x}).
This proves the property for $N=1$.

Let us now assume that the property holds for $N$ and let us prove it for $N+1$.
We have
$$
\abs{(\Phi_1 \Phi^N \knoi)_\ell(x,\xi)} \leq 
\int_{[\sigmanoi_\ell(x,\xi),x]} \sum_{r=1}^{n-1} \abs{(\Phi^N \knoi)_r(s,\zetanoi_\ell(s;x,\xi))} \abs{\mnoi_{r \ell}} \, ds,
$$
Using the induction assumption, we get
$$
\abs{(\Phi_1 \Phi^N \knoi)_\ell(x,\xi)} \leq 
C_1 \frac{C^N}{N!}
\norm{\knoi}_B
\int_{[\sigmanoi_\ell(x,\xi),x]} \abs{s-\zetanoi_\ell(s;x,\xi)}^N \, ds.
$$
Similarly, noting that $\eta-\zeta_i(\eta;s,\zetanoi_\ell(s;x,\xi))=s-\zetanoi_\ell(s;x,\xi)$, we get
$$
\abs{(\Phi_2 \Phi^N \knoi)_\ell(x,\xi)} \leq 
C_2 \frac{C^N}{N!}
\norm{\knoi}_B
\int_{[\sigmanoi_\ell(x,\xi),x]} \abs{s-\zetanoi_\ell(s;x,\xi)}^N \, ds.
$$
Now observe that $\abs{s-\zetanoi_\ell(s;x,\xi)}
\leq C_4 \abs{s-\sigmanoi_\ell(x,\xi)}$ for some $C_4>0$ that does not depend on $\ell,s,x,\xi$.
It follows that
$$
\int_{[\sigmanoi_\ell(x,\xi),x]} \abs{s-\zetanoi_\ell(s;x,\xi)}^N \, ds
\leq C_4 \int_{[\sigmanoi_\ell(x,\xi),x]} \abs{s-\sigmanoi_\ell(x,\xi)}^N \, ds
=C_4 \frac{\abs{x-\sigmanoi_\ell(x,\xi)}^{N+1}}{N+1}.
$$
We conclude thanks to the estimate \eqref{estim x minus sigmax}.

\end{proof}

Finally, the estimate \eqref{thm exist kern estimate} can be deduced from the identities $\knoi=F^N(\knoi)-F^N(0)+F^N(0) =\Phi^N(\knoi)-\Phi^N(0) +\sum_{r=0}^N \Phi^r u^0$, combined with the fact that $\Phi^N$ is a contraction and that $u^0$ can be estimated by the right-hand side of \eqref{thm exist kern estimate} (with $s=0$).
This concludes the proof of Theorem \ref{thm exist kern} for $s=0$.

To prove the result for $s \geq 1$ we can argue as in the proof of \cite[Theorem 3.6]{Bre00} and then use an induction argument.

\begin{remark}
The proof above can be adapted to deal with space-dependent systems, i.e. when $\lambda_j$ and $m_{rj}$ depend on $x$.
The additional condition for $k_i$ has to be modified though, but we can for instance consider $k_i(x,0)=f_i(x)$.
Note that we still have explicit formulas for the corresponding $\zeta_j$ and $\sigma_j$.
\end{remark}

\begin{remark}
Our approach can be used to recover existence results in the triangle $\Tau$.
To this end, we simply extend the parameters $\lambda_j$ and $m_{rj}$ to $[a,b] \supset [0,1]$ in a smooth way.
Then, for $a,b$ large enough, the domain $D$ will contain the triangle $\Tau$ (recall \eqref{basic prop D}) and we apply Theorem \ref{thm exist kern} in this $D$.
This approach is different from all the previous ones in the literature, which consisted in adding ``artificial boundary conditions'' at some parts of the boundary of $\Tau$.
Note in addition that extending $\lambda_j$ and $m_{rj}$ outside $[0,1]$ in a smooth way is easier than building artificial boundary conditions that satisfy compatibility conditions associated with the kernel equations.
\end{remark}

\bibliographystyle{amsalpha}
\bibliography{biblio}

\end{document}